\documentclass[12pt,letterpaper]{article}

\usepackage[pagebackref,bookmarksnumbered,colorlinks,
bookmarks, linkcolor=blue,breaklinks, linktocpage]{hyperref}

\usepackage[applemac]{inputenc}
\usepackage[T1]{fontenc}
\usepackage[english]{babel}
\usepackage{amsmath}
\usepackage{amssymb,amsfonts,textcomp,amsthm,amsmath,varioref}
\usepackage{array}
\usepackage{graphicx}
\usepackage{wrapfig}
\usepackage[rflt]{floatflt}
\usepackage[dvips, dvipsnames, usenames]{color}
\usepackage[numbers]{natbib}
\usepackage[mathcal]{euscript}
\usepackage{mathabx,mathrsfs}

\usepackage{geometry}
\makeatletter
\renewcommand\section{\@startsection{section}{1}{1em}{-\baselineskip}{0.5\baselineskip}{\normalfont\normalsize\bf}}
\makeatother

\makeatletter
\renewcommand\@seccntformat[1]{{\csname the#1\endcsname}{.}\hspace{0.25em}}
\makeatother

\makeatletter
\renewcommand\subsection{\@startsection{subsection}{2}{0em}{-\baselineskip}{0.5\baselineskip}{\normalfont\normalsize\bf}}
\makeatother

\theoremstyle{plain}
\newtheorem{lemma}{Lemma}
\newtheorem{proposition}{Proposition}
\newtheorem{theorem}{Theorem}

\theoremstyle{remark}
\newtheorem{remark}{Remark}
\newtheorem{condition}{Condition}

\theoremstyle{definition}
\newtheorem{example}{Example}
\newtheorem{definition}{Definition}

\newcommand{\HRule}{\rule{\linewidth}{0.25mm}}

\newcommand{\Z}{\mathbb{Z}}
\newcommand{\R}{\mathbb{R}}
\newcommand{\Sph}{\mathbb{S}}

\newcommand{\D}{\mathbb{D}}
\newcommand{\Zt}{\mathbb{Z}_{12}}
\newcommand{\I}{\textrm{Id}}

\newcommand{\Rc}{\mathcal{R}}
\newcommand{\Sc}{\mathcal{S}}
\newcommand{\Tc}{\mathcal{T}}
\newcommand{\Nc}{\mathcal{N}}

\newcounter{PARA}
\newcounter{DFN}

\newenvironment{defin}{\refstepcounter{DFN}\vspace{2mm}\begin{center}\quad\quad\hfill }{\hfill \textsc{
(DEF.~\theDFN)}\end{center} }

\begin{document}
\title{ {\Huge{ A Tonnetz model for pentachords }}}
\author{{{Luis A. Piovan}}}

\date{}
\maketitle


\ 

\noindent{\sc{ABSTRACT.}}
This article deals with the construction of surfaces that are suitable for representing pentachords or $5$-pitch segments that are in the same $T/I$ class. It is a generalization of the well known \"Ottingen-Riemann torus for triads of neo-Riemannian theories. Two pentachords are near if they differ by a particular set of contextual inversions and the whole contextual group of inversions produces a Tiling (Tessellation) by pentagons on the surfaces. A description of the surfaces as coverings of a particular Tiling is given in the twelve-tone enharmonic scale case.

\noindent\HRule 

{\centering{
\section{Introduction}\label{sec:int}
}}

The interest in generalizing the \"Ottingen-Riemann Tonnetz was felt after the careful analysis David Lewin made of Stockhausen's Klavierst\"uck III \citep[][Ch.~2]{lew207}, where he basically shows that the whole work is constructed with transformations upon the single pentachord $\langle C, C\#, D, D\#, F\#\rangle$. A tiled torus with equal tiles like the usual Tonnetz of Major and Minor triads is not possible by using pentagons (you cannot tile a torus or plane by regular convex pentagons). Therefore one is forced to look at other surfaces and fortunately there is an infinite set of closed surfaces where one can gather regular pentagonal Tilings. These surfaces (called hyperbolic) are distinguished by a single topological invariant: the genus or number of holes the surface has (see Figure~\ref{fig:g4})\footnote{Not all genera admit a regular pentagonal Tiling.}.   

The analysis\footnote{Prepint by David Lewin cited in \citep{fisa05}.} of Schoenberg's, Opus 23, Number 3, made clear the type of transformations\footnote{In this work again there is a fundamental generating pentachord $\langle B\flat, D, E, B, C\#\rangle$ and the group of transformations contain the $(T/I)$-class of this pitch segment.} to be used. These are the basic contextual transformations (inversions) we are concerned in this article and they are given in (DEF. \ref{defin:abe}).

The main result of this paper is given in Theorem~\ref{theo}. We propose a surface of a high genus as a topological model for pentachords satisfying certain conditions. We give a construction of this surface in terms of a covering of another surface of smaller genus which carries a pentagonal Tiling on it. The structure of the group of contextual transformations is studied and related to the surface. Many examples of pentachord or $5$-pitch segments (and their $T/I$ classes) including the two examples above fit into this kind of surfaces as a regular Tiling. 

In a recent paper Joseph Straus \citep{str11} gave an interpretation of passages of Schoenberg's Op. 23/3 and  Igor Stravinsky's In Memoriam Dylan Thomas. He uses a different set of transformations on pentachords or $5$-pitch segments: a combination of inversions and permutations. Straus describes a space for these transformations and it would be interesting to check whether this kind of contextual groups fits into our framework.

Also, the citations in \citep[][pag. 49]{je06} are an useful source of examples for studying $5$-pitch segments. For instance, 
A. Tcherepnin uses the different modes of $\langle C, D, E, F, A\rangle$ and $\langle C, D, E\flat, G, A\flat\rangle$ in his Ops. 51, 52, and 53.

We use the dodecaphonic system. This is best suited for Stockhausen or Schoenberg's works and it translates into numbering pitch classes modulo octave shift by the numbers in $\Zt$. However, as we see in the following passages other systems are admissible. For instance, diatonic with numbers in $\Z_7$.  

\begin{example}\label{exa:koechlin}
The piece IV of Les Heures Persanes, Op. 65 by Charles Koechlin contains passages like that of Figure~\ref{fig:koechlin1}, where starting with a given pentachord he makes several parallel pitch translates but in a diatonic sense. 
In this case you have to assign to the notes a number in $\Z_7$ (for instance, $[C,D,E,F,G,A,B] \rightarrow [0,1,2,3,4,5,6]$). There are $14$  $(T/I)$-forms  for any pentachord.
\begin{figure}
 \centering
    \includegraphics[width=\textwidth]{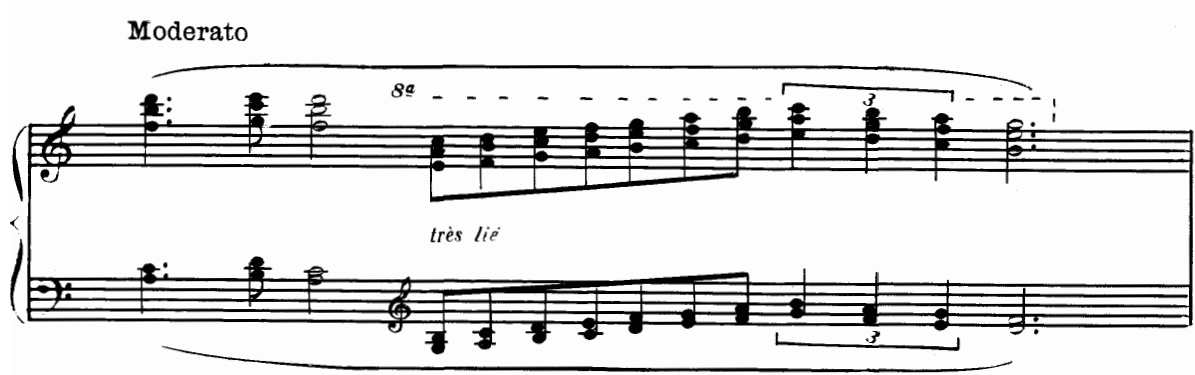}
 \caption
 {Excerpt of Les Heures Persanes, Op. 65, IV. Matin Frais dans la Haute Vall\'ee by Charles Koechlin. Copyright 1987 by Editions Max Eschig.}\label{fig:koechlin1}
\end{figure}
\end{example}  
\begin{example}\label{exa:stravinsky}
The piano part of Figure~\ref{fig:petrushka} from the Ballet Petrushka by Igor Stravinsky also moves pentachords in a parallel diatonic way. The same observations as in Example~\ref{exa:koechlin} apply.
\begin{figure}
 \centering
    \includegraphics[width=\textwidth]{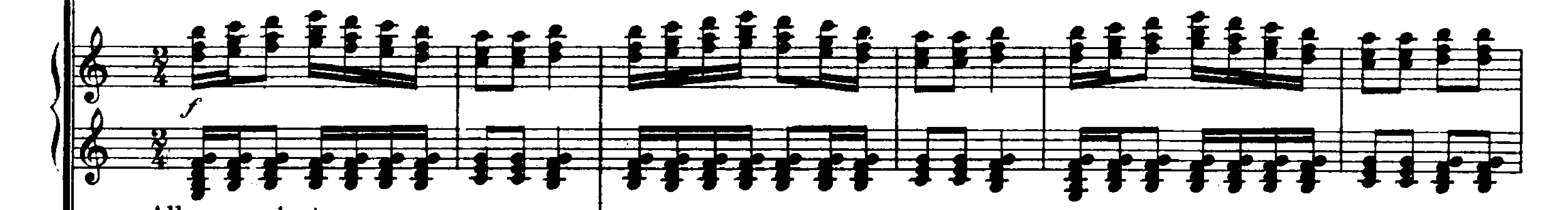}
 \caption
 {Piano part of Russian Dance, Part I (No. IV) of the Ballet Petrushka by Igor Stravinsky: Petrushka, the Moor and the Ballerina suddenly begin to dance, to great astonishment of the crowd.}\label{fig:petrushka}
\end{figure}
\end{example}  
In this article we are not going to pursue the case of a pitch system with numbering in a general $\Z_m$. Our approach requires to study each numbering system on a case by case basis. Although, I presume a theory over any $\Z_m$ is possible, but with heavier machinery. 

In Section~\ref{sec:pitch} we introduce the contextual inversions and translation operators. We give two lemmas relating these operators. Several examples are developed in cases of triads and tetrachords. Tessellations are introduced. In Section~\ref{sec:triangle} surfaces and Triangle groups (groups related to Tilings of a given surface) are explored. Several formulas are given on how to compute the genus in terms of a  Tiling of type $\{F,p,q\}$ ($F$ regular convex $p$-gons where $q$ polygons meet at each vertex).  Section~\ref{sec:incidence} is devoted to describe the structure of  the transformation group. This group fits into an exact sequence with two terms: a quotient isomorphic to the dihedral $\D_{12}$ (or a subgroup of it) that maps into the triangle group of a surface with minimal Tiling, and an abelian subgroup. Our transformation group is also a semi-direct product of an inversion operator and an abelian group of translation operators.  Section~\ref{sec:surface} contains the main theorem that gives the construction of a tiled surface by pentagons so that all transforms of a given pentachord fit together. This surface is a covering of a genus $13$ surface, which is in turn  a $2-1$ cover of the genus $4$ Bring\footnote{Surface already introduced in the $19^{\text{th}}$ century that can be regularly tiled into $12$ pentagons.} surface. 

{\centering{
\section{Tone networks}\label{sec:ton}
}}


Tone networks (Tonnetz) where invented by Leonhard Euler as a way of visualizing harmonically related tones by means of a graph where points represent pitches, or by its dual graph, where points represent tones. The musicologist Hugo Riemann extensively used these networks in his theory. More recently, group theorists following David Lewin and Richard Cohn related these ideas with $LPR$ operations \citep[see][for a historical perspective]{co98}. When we assume enharmonic equivalence and equal-tempered tuning the familiar \"Ottingen-Riemann Tonnetz associated with Major and Minor triads becomes double periodic and the graph can be wrapped around a torus ${\mathbb T}^2$ in a regular way. That is, the torus surface splits into equilateral (curved) triangle tiles whose sides are the edges of the    graph and all of its vertices have 6 incident edges. Such an object is known as a regular tessellation of type $\{3,6\}$ on the torus. 

In a series of papers, Richard Cohn \citep{co97,co03} relates the geometry of certain Tonnetze with a kind of voice leadings which he calls ``parsimonious''. This is moving from one tone to a contiguous tone on the torus Tonnetz. Although this kind of voice leadings seem to have more to do with acoustic or harmonic proximity \citep{ty09} rather than motion by a few close notes at a time, Tonnetze allow us to decompose Tone space into different fibers: the tori. This brings a sort of better understanding into the geometry of this complex space. The geometry of Tonnetze is spelled out by a set of transformations called contextual transformations (the group generated by $LPR$ operations in case of the \"Ottingen-Riemann Tonnetz). It was shown by Lewin in his analysis of Schoenberg's Opus 23, No. 3, and in general by Fiore and  Satyendra \citep{fisa05} (with a slightly different definition of contextual group) that a quotient of the group of contextual transformations is precisely isomorphic to the dual of the 24 $T-$ and $I-$forms of any pitch segment $\langle x_1, \dots, x_n\rangle$ \citep[see also][]{crfisa09}\footnote{This paper has a nice account for the dictionary between notes and the integers $\mod 12$.}.

{\centering{
\section{Pitch classes, Operators and Tilings}\label{sec:pitch}
}}

Our main objects will be pitch classes of length $n$. These classes will be vectors whose entries are real numbers representing pitches modulo $12 \mathbb{Z}$. A pitch shift of 12 is an octave shift and so pitches live in a circle of length 12 (where $0\equiv 12$) and pitch classes of length $n$ are elements of an $n-$dimensional torus $\mathbb{T}^n=(\mathbb{R}/12\mathbb{Z})^n$. That is, we regard only the octave shift symmetry ``O'' from the ``OPTIC'' set of symmetries\footnote{The vectors will be ordered classes. However, later on we will have to consider the ``OP'' equivalence, that is octave shift and permutations of the vector entries. This is the symmetrization of the torus $\mathbb{T}^n$: $S(\mathbb{T}^n)=\mathbb{T}^n/S_n$.} considered in \citep{caquty08}, \citep[see also][]{ty11}. Most of the time we will be concerned with the set of integer pitches $\{0, 1, 2, 3, 4, 5, 6, 7, 8, 9, 10, 11\}=\mathbb{Z}_{12}$ which are in one to one correspondence with the notes of the chromatic scale $\{C, C\#=D\flat, D, D\#=E\flat, E, F, F\#=G\flat, G, G\#=A\flat, A, A\#=B\flat, B\}$\footnote{For short we will denote $t=10$, $e=11$.}. These twelve number classes allow addition and multiplication operations as usual numbers do, but always computing the positive remainder of the division by 12. For instance $(-7).5=5.5=25=1$; $10+7=17=5$. 

\subsection{Main Definitions}\label{parala:def}
The usual translation and inversion operators are defined in $\mathbb{Z}_{12}$: $T_n(x):=x+n$, with $n\in\Zt$, $x\in\Zt$ and $I(x):=-x$ with $x\in\Zt$. Also, $I_n(x):=T_n\circ I (x)=-x+n$, $n\in\Zt$ and $x\in\Zt$ is the inversion around $n$. Remember that when making computations with these operators all of them should be carried out in $\Zt$. The group generated by $\{I,T_1\}$ (the $T/I$ group) is isomorphic to the dihedral group $\D_{12}$. His elements are $\{T_0, T_1, \dots, T_{11}, I_0=I, I_1, \dots, I_{11}\}$ and the following relations hold $I^2=\I, \ T_1^{12}=\I, \ I\circ T_1 \circ I = T_1^{-1}$. Here $\I=T_0$ is the identity operator and the power $T_1^m=T_m$ means compose $T_1$ $m$ times.

Each of these operators are defined component-wise on vectors: $T_m \langle x_1, x_2, \dots, x_n\rangle:=\langle T_m(x_1), T_m(x_2), \dots, $ $T_m(x_n)\rangle$, and $I_m \langle x_1, x_2, \dots, x_n\rangle:=\langle I_m(x_1), I_m(x_2), \dots, I_m(x_n)\rangle$. On segments of length $n$ we define our contextual operators to be the operators $p_{ij}$ with $i\neq j$, $1\leq i, j \leq n$, defined as follows\footnote{These definition  already appears in \citep{fisa05} but comes from earlier work by Lewin et al.}:
\begin{defin}\label{defin:abe}
$p_{ij}\langle x_1, x_2, \dots, x_n\rangle:=I_{x_i+x_j}\langle x_1, x_2, \dots, x_n\rangle.$
\end{defin}
All of them have order 2, namely $p_{ij}^2=\I$, and they satisfy some relations that we will state later. 

We also consider the permutation operators $\tau_{ij}$ and $\sigma$, $1\leq i < j \leq n$. $\tau_{ij}$ transposes only the $i$ and $j$ coordinates of the vector while $\sigma$ ciclicaly permutes each coordinate:
$$\tau_{ij}\langle x_1, \dots, x_i, \dots, x_j, \dots, x_n\rangle:=\langle x_1, \dots, x_j, \dots, x_i, \dots, x_n\rangle.$$
$$\sigma\langle x_1, x_2, \dots, x_{n-1}, x_n\rangle:=\langle x_2, x_3, \dots, x_n, x_1\rangle.$$
The full symmetric group $S_n$, which has $n!$ elements, is generated by $\tau_{i,i+1}$ and $\sigma$, for any $i\geq 1$. Other sets of generators are the following $\{\tau_{12}, \tau_{23}, \dots, \tau_{n-1,n}\}$; $\{\tau_{12}, \tau_{13}, \dots, \tau_{1n}\}$; and any set containing a 2-cycle and an n-cycle.

\subsection{Some Lemmas}\label{parala:as}
In order to prove statements  about our contextual operators we find useful to represent them by matrices with entries in $\Zt$. Operations with matrices should be carried out in $\Zt$ and statements about matrices will transfer to statements about the elements of the group. So the contextual operators (as well as the elements of $S_n$) can be thought of as linear operators while the elements of the $T/I$ group are in the affine group. Namely, $T_m$ is translation by the vector $\langle m, m, \dots, m\rangle$ and $I_m$ is inversion plus translation by $\langle m, m, \dots, m\rangle$. 

We define a new linear operator. For any set of four indexes  $i, j, h, k$, $1\leq i, j, h, k\leq n$ let 
\begin{defin}\label{defin:bee}
$T^{hk}_{ij}\langle x_1, x_2, \dots, x_n\rangle:=T_{x_h+x_k-x_i-x_j}\langle x_1, x_2, \dots, x_n\rangle.$
\end{defin}
In case one of the upper indices coincides with a lower index we cancel them and write for short $T^k_j=T^{ik}_{ij}$.

\begin{lemma}\label{lem:first}
The contextual operators $p_{ij}$, $1\leq i < j \leq n$, satisfy the following relations.
\begin{align}
p_{ij}\circ p_{hk}&=T^{hk}_{ij}, &\textrm{for any } i, j, k, l, \label{eq:first}\\
\sigma^{i-1}\circ p_{i,i+1}&= p_{12}\circ\sigma^{i-1}, &2\leq i, \label{eq:second}\\
\tau_{ij}\circ p_{hj}&=p_{hi}\circ \tau_{ij}, &h< i<j,    \\
\tau_{2i}\circ p_{1i}&=p_{12}\circ \tau_{2i}, &3\leq i. \label{eq:four}
\end{align} 
\end{lemma}
\begin{proof}First notice that by (DEF.~\ref{defin:abe}) we have $p_{ij}=p_{ji}$. So we consider those $p_{ij}$ with $i<j$. We have that $p_{ij}\circ p_{hk}\langle x_1, x_2, \dots, x_n\rangle=p_{ij}I_{x_h+x_k}\langle x_1, x_2, \dots, x_n\rangle$. Now $\lambda=I_{x_h+x_k}(x_i)+I_{x_h+x_k}(x_j)=2 (x_h+x_k)-(x_i+x_j)$, and applying $I_{\lambda}$ to an element $I_{x_h+x_k}(x_r)=x_h+x_k-x_r$ we get $x_h+x_k-x_i-x_j+x_r$. This shows (1). 

Compute $p_{12}\sigma^{i-1}\langle x_1, x_2, \dots, x_n\rangle$. Evaluated on each coordinate $x_r$ gives $\sigma^{i-1}(x_1)+\sigma^{i-1}(x_2)$ $-\sigma^{i-1}(x_r)=x_i+x_{i+1}-\sigma^{i-1}(x_r)$, and this is precisely the left hand operator in (2) when evaluated on the coordinate $x_r$. This proves (2) in case $i<n$. If $i=n$ we make the convention $i+1=1$ and the result also hods.

Since $\tau_{ij}$ exchanges the entries $x_i$ and $x_j$, the right hand side of (3) on coordinate $x_r$ takes the value $x_h+x_j-x_r$ if $r$ is different from either $i$ or $j$. On $x_i$ the right hand side is equal to $x_h$ and on $x_j$ is equal to $x_h+x_j-x_i$. An easy checking gives the same values on the left hand side of (3). Finally (4) is a consequence of (3).
\end{proof}

\begin{lemma}\label{lem:second} The set of operators $p_{ij}$'s and $T^{hk}_{ij}$'s satisfy the following with respect to a translation $T_m$ and inversion $I$.
\begin{align}
p_{ij}\circ T_m&=T_m\circ p_{ij}          &{\text{for any }} i, j, \text{ indices and } m\in\Zt,\\
p_{ij}\circ I&=I\circ p_{ij}             &\textrm{for any indices } i, j,  \\
T^{hk}_{ij}\circ T_m&=T_m\circ T^{hk}_{ij} &{\text{for any }} i, j, h, k, \text{ indices and } m\in\Zt,\\
T^{hk}_{ij}\circ I&= I\circ T^{hk}_{ij} &\textrm{for any indices } i, j, h, k.
\end{align}
Thus, any subgroup generated by a set of $T^{hk}_{ij}$'s is abelian. Moreover, $T^{hk}_{ij}=T^{h}_{i}\circ T^{k}_{j}$, $(T^{hk}_{ij})^{-1}=T^{ij}_{hk}$, and the subgroup generated by any $p_{uv}$ and $T^{hk}_{ij}$ is isomorphic to the dihedral group $\D_{12}$.
\end{lemma}

\begin{proof} To show (5) we check the formula for $m=1$. $p_{ij}T_1\langle x_1, x_2, \dots, x_n\rangle=p_{ij}\langle x_1+1, x_2+1, \dots, x_n+1\rangle=\langle x_i+x_j-x_1+1,x_i+x_j- x_2+1, \dots, x_i+x_j-x_n+1\rangle=T_1p_{ij}\langle x_1, x_2, \dots, x_n\rangle$. Formula (6) is also straight forward $p_{ij}I\langle x_1, x_2, \dots, x_n\rangle=p_{ij}\langle -x_1, -x_2, \dots, -x_n\rangle=\langle -x_i-x_j+x_1, -x_i-x_j+x_2, \dots, -x_i-x_j+x_n\rangle=Ip_{ij}\langle x_1, x_2, \dots, x_n\rangle$.   By virtue of formula~\eqref{eq:first} we get (7) and (8) from (5) and (6). Any subgroup generated by a set of $T^{hk}_{ij}$'s is abelian because of (DEF.~\ref{defin:bee}) and (7). Also, since $p_{uv}$ has order $2$ and $T^{hk}_{ij}$ has order $12$ in general\footnote{We make the distinction between $T^{hk}_{ij}$ applied to a generic pitch class and to a particular $(T/I)$ set generated by a fixed  pitch class $\langle x_1, x_2, \dots, x_n\rangle$. In the first case, since any values for the entries of the vector can occur (in particular, a vector for which $x_h+x_k-x_i-x_j=1$), we get that $(T^{hk}_{ij})^{12}=\I$. As for the second case, the order of $T^{hk}_{ij}$ can be any divisor of $12$, namely $2, 3, 4, 6,$ or $12$.}, it will suffice to prove that $p_{uv}T^{hk}_{ij}p_{uv}=(T^{hk}_{ij})^{-1}$. Now $p_{uv}T^{hk}_{ij}p_{uv}=p_{uv}p_{ij}p_{hk}p_{uv}=T^{ij}_{uv}T^{uv}_{hk}=T^{ij}_{hk}=(T^{hk}_{ij})^{-1}$
\end{proof}

\subsection{Coxeter groups}\label{parala:coxeter}
As remarked in [\ref{parala:def}] after (DEF.~\ref{defin:abe}), the operators $p_{ij}$'s are symmetries. However, because of equation~\eqref{eq:first} they satisfy $(p_{ij}p_{hk})^{12}=\I$ since $(T^{hk}_{ij})^{12}=T^{12}_{x_h+x_k-x_i-x_j}=T_{12(x_h+x_k-x_i-x_j)}=T_0$. So at least, the group generated by these contextual transformations is a quotient of the Coxeter group abstractly\footnote{This means that we are regarding the $p_{ij}$'s as the generators of a Coxeter group merely with the relations  satisfied by Coxeter groups and no extra relation.}  generated by  $p_{ij}$'s and with relations $p_{ij}^2=(p_{ij}p_{kl})^{12}=\I$.
\setcounter{definition}{2}
\begin{definition}
A Coxeter group is a group defined by a set of generators $\{R_1, R_2, \dots, R_s\}$ and relations $\{R_i^2=(R_i R_j)^{m_{ij}}=\I$, $1\leq i \leq s$, $m_{ij}\geq2$ positive integers for $i<j\}$. 
\end{definition}
Coxeter groups are usually infinite but have nice geometric realizations as reflections about hyperplanes in $n$-dimensional Euclidean spaces \citep[][chap. 9]{coxmos80}. Also, some of these groups are presented with two generators (one of them a reflection), and in these cases the target space is a surface \citep[][chap. 8]{coxmos80}. 

The symmetric group $S_n$ can be presented with different sets of generatos and relations \citep[][chap. 6]{coxmos80} and in particular as a quotient of a Coxeter group in the generators $\{\tau_{12}, \tau_{23}, \dots, \tau_{n-1,n}\}$.

\subsection{Regular Tilings}\label{parala:polygon}
$LPR$ operations on a triadic segment are just the operations induced by the group $G_3$ generated by $\{p_{12}, p_{13}, p_{23}\}$ \citep{crfisa09}. The \"Ottingen-Riemann torus is constructed by associating with the segment $\langle x_1, x_2, x_3 \rangle$ an oriented equilateral triangle with consecutive  vertices $\{1, 2, 3\}$ in correspondence with the notes of the segment. Then one reflects this ``Tile'' along the side $ij$ (meaning reflecting by $p_{ij}$) and glues the two tiles along the common side. One continues in this way by applying all the elements  of $G_3$ and hopes for the tiles to match together into a surface $S$. 

A Tiling on a surface $S$ by a convex polygon $T$ (for instance a regular polygon) is the action of a group on $S$ such that you can cover the surface $S$ with translations of $T$ by elements of the group in such a way that two tiles are either disjoint, they meet at a common vertex or they meet along a common side (perfectly interlocking with each other). ``Tilings'' or ``Tessellations'' with a pattern Tile $T$ on a surface $S$ are not always possible; they depend on the shape of the tile and the topology of the surface. Tilings on the plane $\R^2$ or the sphere $\Sph^2$ (polyhedra)  described in a mathematical fashion were basically initiated by Kepler in his book \citep{ke19}. Tilings in the plane with different shapes, patterns and group of symmetries are thoroughly studied in \citep{grsh86} and polyhedra in \citep{cox74}. A  more recent account on Tessellations and Symmetries is in the beautiful book by Conway et al. \citep{cobugo08}. 

The problem of finding a Tiling on a surface $S$ by a number $F$ of equal regular (possibly curved) polygons with $p$ sides such that at each vertex the same number $q$  of (incident) polygons meet is called the problem $\{F, p, q\}$ on $S$.  When there is no regarding of the number $F$ of tiles or faces we call it simply the problem $\{p, q\}$ on $S$. Solving this problem depends on the group of automorphisms that a surface $S$ admits and an attempt to classifying them is done in a series of papers (at least up to genus 101) by Conder et al. \citep{condob01,con06}. For instance, as well known, the only possible Regular Tessellations on the sphere $\Sph^2$ are the usual regular simple polyhedra: 
\begin{enumerate}
\item the regular {\bf tetrahedron} $\{3,3\}$, with group of symmetries of order 12: $A_4$ (the even permutations in $S_4$),
\item the {\bf cube} $\{4,3\}$ and the {\bf octahedron} $\{3,4\}$, with group of symmetries of order 24: $S_4$, 
\item the {\bf dodecahedron} $\{5,3\}$ and the {\bf icosahedron} $\{3,5\}$, with group of symmetries of order 60: $A_5$ (the alternating group of $S_5$).
\end{enumerate}

In general, if a finite Tessellation on a surface $S$  is found, the number of faces divides the order of the group of symmetries.
On the torus the only possible Regular Tessellations are of type $\{4,4\}$, $\{6,3\}$ and $\{3,6\}$. In this case any number of faces are allowed as long as they close up to a torus.

 An exploration of the  possibility of drawing Regular Tilings in Computer Graphics is given in the paper by van Wijk \citep{van09}.

\subsection{Trichords}\label{parala:g3}
Now we deduce the abstract structure of the group $G_3$ defined in [\ref{parala:polygon}] in order to explain
 the construction of the torus. By formula~\eqref{eq:second}  we have $p_{23}=\sigma^{-1}p_{12}\sigma$
 \begin{figure}
 \centering
    \includegraphics[width=0.30\textwidth]{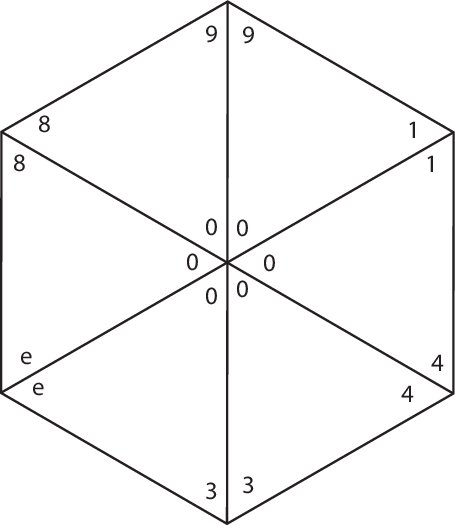}
 \caption{incidence for the segment $\langle0, 1, 9\rangle$ at vertex $0$}\label{fig:hexag}
\end{figure}
  and $p_{13}=\sigma^{-2}p_{12}\sigma^2$. Therefore it is natural to map $G_3$ into the group generated by $\{p_{12}, \sigma\}$. In order to say something about this new group we try to find relations between $p_{12}$ and $\sigma$. Write them as matrix operators: $p_{12}=\left(\begin{smallmatrix} 0&1&0\\1&0&0\\1&1&-1\end{smallmatrix}\right)$, $\sigma=\left(\begin{smallmatrix} 0&1&0\\0&0&1\\1&0&0\end{smallmatrix}\right)$. 
Then $z=\sigma p_{12}=\left(\begin{smallmatrix}1&0&0\\1&1&-1\\0&1&0\end{smallmatrix}\right)$, and one checks by computing successive powers of $z$ that $z^6=\I$. So the group generated by $\{p_{12}, \sigma\}$ is contained in the ``rotation'' group: $\langle A, B\ |\ A^2=\I, B^3=\I, (B^{-1}A)^6=\I\rangle$. Such groups appear precisely as generator groups of Regular Tilings. The introduction of $\sigma$ together with the contextual transformations (inversions) $p_{ij}$'s makes sense geometrically since $\sigma$ is a rotation of the triangle vertices. 


We make the following algebraic convention:  glue the triangle associated to the ordered vector $\langle x_1, x_2, x_3\rangle$  with the triangle associated with $\langle x_1, x_4, x_2\rangle$ along the side $\langle x_1, x_2\rangle$, if and only if $x_4=I_{x_1+x_2}(x_3)$. 
\begin{example} We show in the following example the successive tiles about the first vertex by applying $z$: $\langle0, 1, 9\rangle\xrightarrow{z}\langle0, 4, 1\rangle\xrightarrow{z}\langle0, 3, 4\rangle\xrightarrow{z}\langle0, e, 3\rangle\xrightarrow{z}\langle0, 8, e\rangle\xrightarrow{z}\langle0, 9, 8\rangle\xrightarrow{z}\langle0, 1, 9\rangle$. Pictorically we see this in Figure~\ref{fig:hexag}. 

By reflecting along the different sides of the triangles we can complete  the $24$  tiles of the $T/I$ class for $\langle0, 1, 9\rangle$. In this case $p_{13}p_{23}=T^2_1$ acts as $T_1$ (translation by $1$). So, by acting with $G_3$ on the given pitch segment, we get the full set of translated and inverted pitch classes associated with   $\langle0, 1, 9\rangle$.

A minimal Tiling containing the $24$ elements of the $T/I$ translates of $\langle0, 1, 9\rangle$ is pictured in Figure~\ref{fig:small}. 
\end{example}

\begin{figure}
 \centering
    \includegraphics[width=0.60\textwidth]{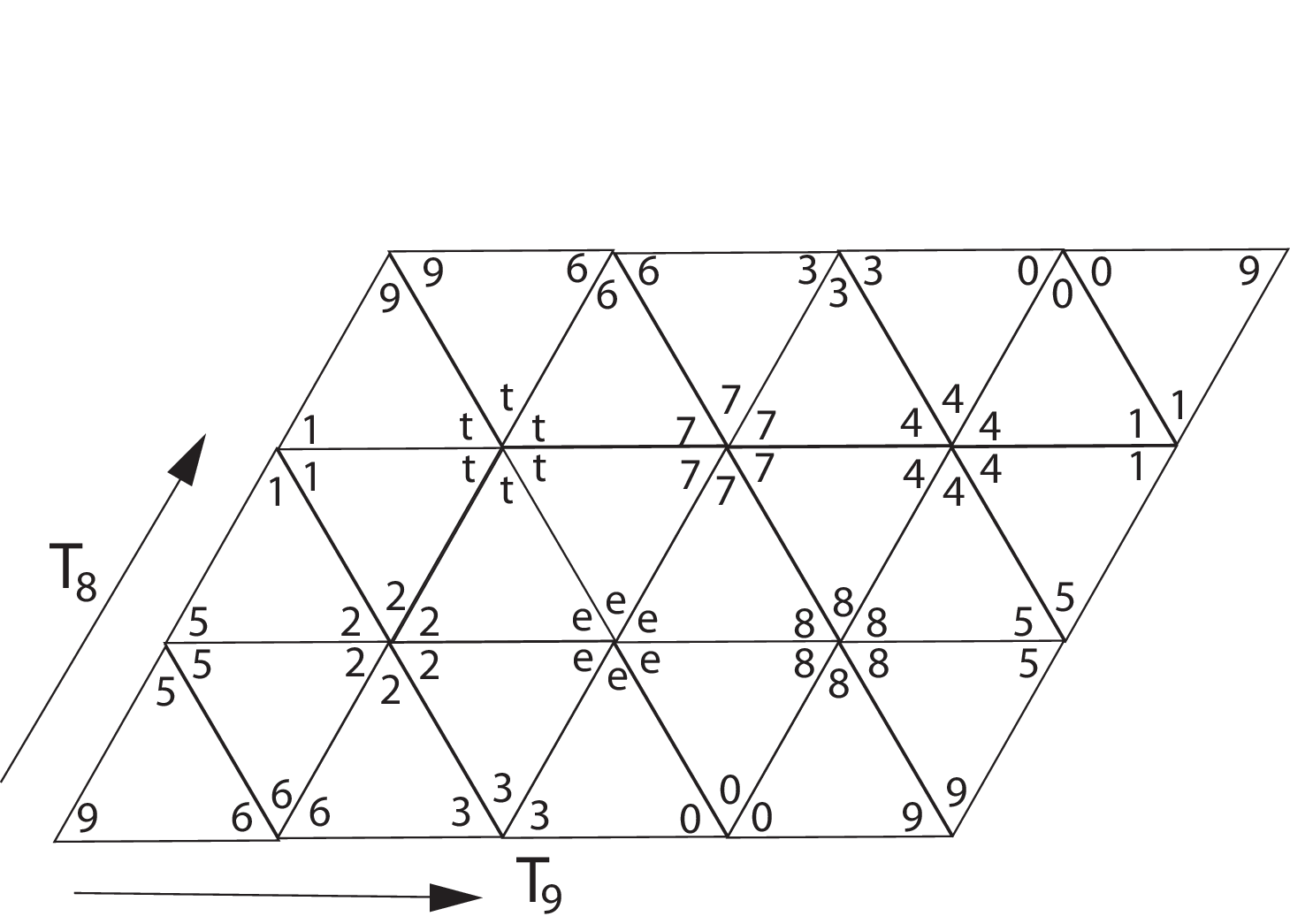}
 \caption{minimal torus for $\langle0, 1, 9\rangle$}\label{fig:small}
\end{figure}

In this example $p_{12}p_{23}=T^3_1\equiv T_9$, and $p_{12}p_{13}=T^3_2\equiv T_8$. By lemma~\ref{lem:second},   $T_8$ and  $T_9$ commute. They are the usual translations on the torus of Figure~\ref{fig:small} moving a unit in each direction. Now, the subgroup generated by $T_8$ and $T_9$ is $\Z_3\oplus\Z_4\simeq\Z_{12}$, and since $G_3$ (restricted to the pitch segment) is generated by $\{p_{12}, T_8, T_9\}$, lemma~\ref{lem:second}    tells us that $G_3$ is isomorphic to the dihedral group $\D_{12}$. Therefore Figure~\ref{fig:small} is a geometrical representation of $G_3\bullet\langle0, 1, 9\rangle\cong(T/I)\bullet\langle0, 1, 9\rangle$\footnote{All these considerations should be possible (with little modification in lemmas~\ref{lem:first},~\ref{lem:second} and definitions~\ref{defin:abe},~\ref{defin:bee}) for pitch classes generated by translation and  $\tau I$ inversion, where $\tau\in S_n$ is a trasposition: the $(T/\tau I)$ group. For instance if $\tau=R$ is the retrograde then we will have the $RI$-chains. Geometrical representations with $RI$-chains involved are in the paper by Joseph Straus  \citep{str11}.}.

\subsection{Tetrachords}\label{parala:g4}
A similar  analysis for the group $G_4$ is more involved.  
This group is generated by $\{p_{12}, p_{13}, p_{14},$ $ p_{23}, p_{24}, p_{34}\}$. An alternative set of generators is $\{p_{12}, T^3_2, T^4_2, T^3_1,  T^4_1, T^{34}_{12}\}$. So, basically we have $\Z_2$ generated by $p_{12}$ and the abelian subgroup generated by $\{T^3_2, T^4_2, T^3_1,  T^4_1\}$. 

$T^{34}_{12}$ is ruled out since it is  the sum of $T^3_1$ and $T^4_2$: for these abelian operators which act   as translations the composition means sum of translations. 

Now, these generators share relations and we would like to find a linearly independent set of generators.


We immediately have the relation $T^3_2-T^4_2=T^3_1-T^4_1=T^3_4$. Therefore, we are left with three generators $\{T^3_1, T^3_2, T^4_1\}$ which in general are linearly independent. Indeed, let $\lambda_1, \lambda_2,$ $\lambda_3$ be numbers in $\Zt$ such that $\lambda_1 T^3_1+\lambda_2 T^3_2+\lambda_3 T^4_1=0$, that is $\lambda_1 (x_3-x_1)+\lambda_2 (x_3-x_2)+\lambda_3 (x_4-x_1)=0$ for any triple of $x_i$'s in $\Zt$. Clearly the only possible $\lambda_i$'s are equal to $0$. If however we work with the $(T/I)$ set generated by a fixed segment $\langle x_1, x_2, x_3, x_4\rangle$, then relations will appear and the group will be smaller (a subgroup of $\Z_2 \ltimes (\Zt \oplus \Zt \oplus \Zt)$).

We cannot map $G_4$ into a rotation group as in the case of $G_3$. Indeed, we have by formulas~\eqref{eq:second} and~\eqref{eq:four} that $p_{23}=\sigma^{-1}p_{12}\sigma$, $p_{34}=\sigma^{-2}p_{12}\sigma^2$,  $p_{14}=\sigma^{-3}p_{12}\sigma^3$, $p_{13}=\tau_{23}p_{12}\tau_{23}$, and $p_{24}=\sigma^{-1}p_{13}\sigma$. Namely, besides the cycle $\sigma$ and symmetry $p_{12}$ one has to introduce the transposition $\tau_{23}$. 

Instead, we consider the subgroup $\widetilde G_4$ generated by $\{p_{12}, p_{23}, p_{34}, p_{14}\}$ which maps into the group generated by $\{p_{12}, \sigma\}$. Representing these as matrix operators $p_{12}=\left(\begin{smallmatrix} 0&1&0&0\\1&0&0&0\\1&1&-1&0\\1&1&0&-1\end{smallmatrix}\right)$, $\sigma=\left(\begin{smallmatrix} 0&1&0&0\\0&0&1&0\\0&0&0&1\\1&0&0&0\end{smallmatrix}\right)$, one gets $z=\sigma p_{12}=\left(\begin{smallmatrix}1&0&0&0\\1&1&-1&0\\1&1&0&-1\\0&1&0&0\end{smallmatrix}\right)$, with $z^4=\I$. 

\begin{floatingfigure}[r]{0.63\textwidth}
 \vspace{3mm}
 \centering
    \includegraphics[width=0.60\textwidth]{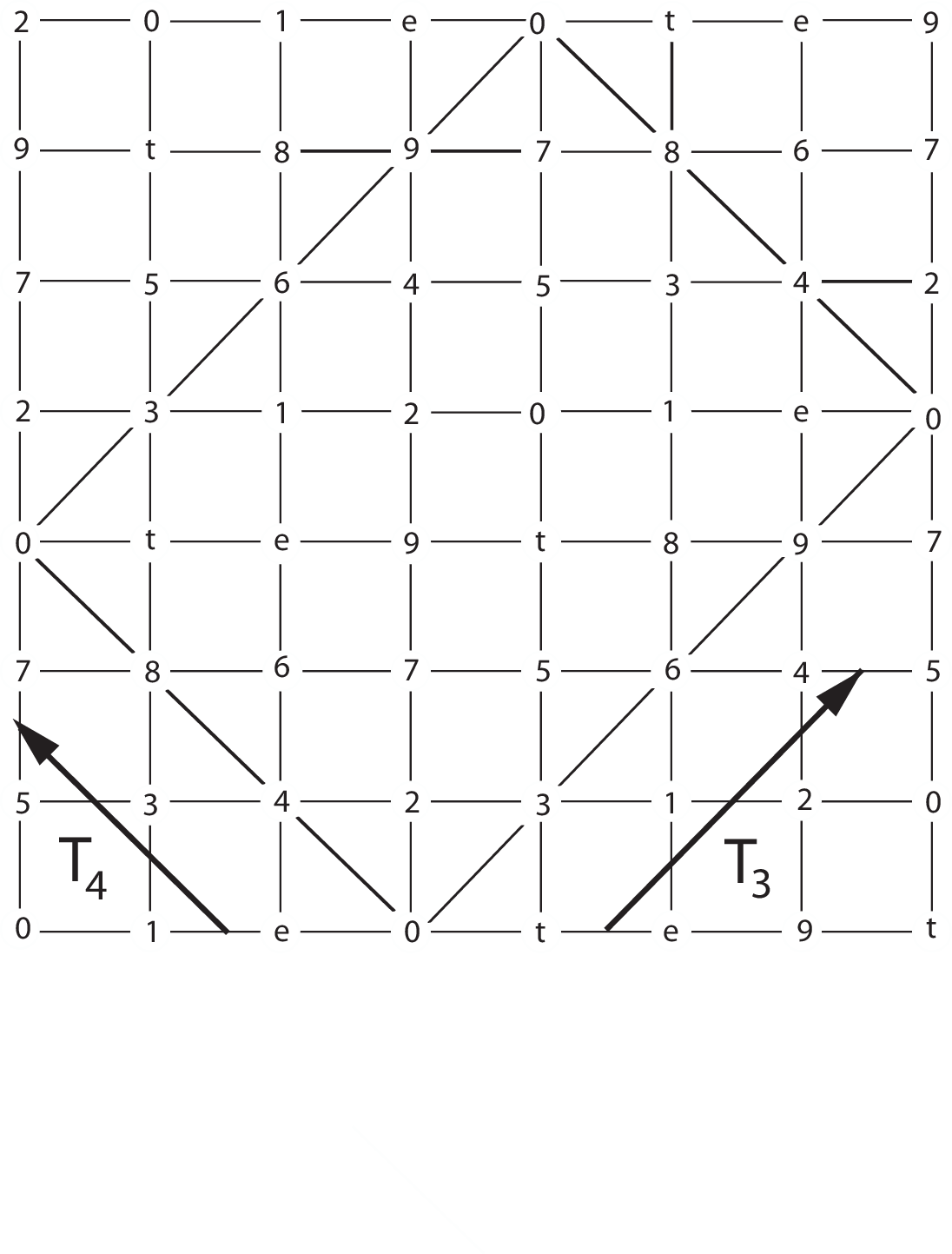}
  \vspace{-3.3cm}   
 \caption{minimal torus for pitch class $\langle0, 1, 3, 5\rangle$}\label{fig:tetra}
\end{floatingfigure}

This corresponds to the rotation group for a Tiling on the plane by squares. Here $\sigma$ is rotation around the square center and $z$ around each vertex. In the general case $\widetilde G_4$ is generated by $\{p_{12},$ $T^3_1, T^4_2\}$, with $T^3_1, T^4_2,$ linearly independent (namely $\widetilde G_4\cong\Z_2\ltimes(\Zt\oplus\Zt)$).

\ 

\noindent{\bf Example 2.}
Restricting to the pitch class $\langle0, 1, 3, 5\rangle$ we have the translations generating the abe-lian subgroup of $\widetilde G_4$: $T^3_1\equiv T_3$ and $T^4_2\equiv T_4$. This is $\Z_4\oplus\Z_3\simeq\Zt$. Figure~\ref{fig:tetra} shows the minimal torus which is generated by the action of  $\widetilde G_4$. Also in here, the whole set class $sc(0135)$ is obtained. The action along the sides of the square tiles by the generators of $\widetilde G_4$ is transversal to the action by $\{T_3, T_4\}$. Notice that these translations have real meaning in terms of  numbers and that they define where the Torus sits.

\ 
However, there may be degeneracies to this model. For instance, consider the following.

\setcounter{example}{2}
\begin{example}
The segment $\langle0, 3, 6, 9\rangle$ whose $(T/I)$ orbit possess several symmetries. Here, the action of $\widetilde G_4$ just repeats four squares and its abelian subgroup is generated by $(T_6,T_6)$. We get a torus out of eight square tiles but not the whole orbit $(T/I)\bullet$$\langle0, 3, 6, 9\rangle$ (see Figure~\ref{fig:failure}). Even the abelian subgroup of the full $G_4$ generated by $(T_6, T_3, T_9)$ would not give the whole $T/I$ orbit. In this case we get three connected components:  $\widetilde G_4\bullet$$\langle0, 3, 6, 9\rangle$, $\widetilde G_4\bullet$$\langle1, 4, 7, t\rangle$, and $\widetilde G_4\bullet$$\langle2, 5, 8, e\rangle$, whose union is the whole set class $sc(0369)$. 
\end{example}

{\centering{
\section{Triangle groups and Surfaces}\label{sec:triangle}
}}

It was said in  [\ref{parala:coxeter}] that the group generated by the contextual transformations $p_{ij}$'s is a quotient of a Coxeter group. We tie this with the Triangle groups that we consider as a special kind of Coxeter groups which have geometrical realizations as Tilings on surfaces.

\begin{definition} given three positive integers $h, k, l\geq2$, A triangle group $\Delta_g(h,k,l)$ is a group generated by three reflections $R_1, R_2, R_3$ satisfying the relations $R_1^2=R_2^2=R_3^2=(R_1 R_3)^h=(R_3 R_2)^k=(R_2 R_1)^l=\I$, and possibly other relations associated with a surface $S_g$. 
\end{definition} 
 \begin{floatingfigure}[r]{0.43\textwidth}
 \centering
    \includegraphics[width=0.40\textwidth]{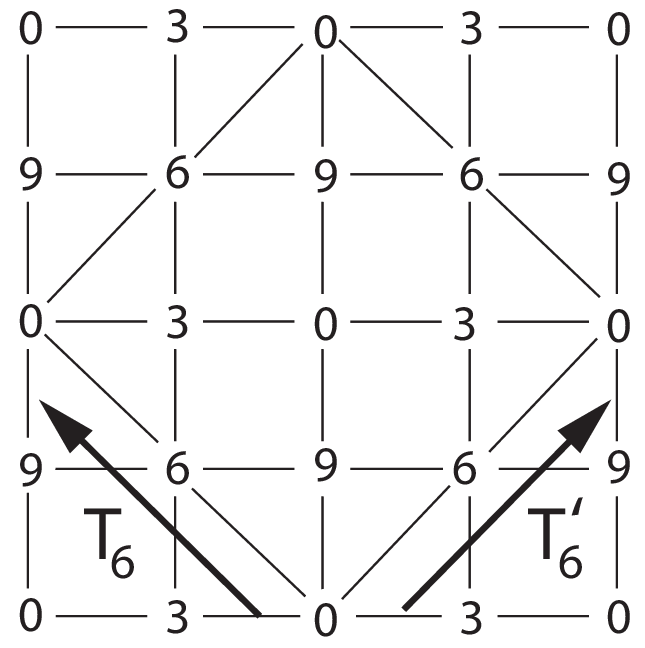}
 \caption{minimal torus for $\langle0, 3, 6, 9\rangle$}\label{fig:failure}
\end{floatingfigure}
These are groups of reflections of regular triangulations on surfaces, that is triangulations associated with Regular Tessellations. Any surface can be subdivided by triangles but not every surface contains a Regular Tessellation by polygons with a certain number of faces (problem $\{F,p,q\}$ cited in [\ref{parala:polygon}]). For instance, if a surface $S_g$ contains a Regular Tiling by squares\footnote{This only happens if the surface is a Torus or the plane.}, then each square is subdivided into triangles with vertices at the center and middle edges as shown in Figure~\ref{fig:square}. There is a fundamental triangle $T$ (yellow) and its mirror reflections (red, blue and green) about the edges. 
The composition of two of these reflections gives a rotation around one of the black hinges  (the vertices of the triangle). This system of rotations is called a rotation group (as mentioned in [\ref{parala:g3}]). 

\begin{floatingfigure}[r]{0.43\textwidth}
 \vspace{3mm}
 \centering
    \includegraphics[width=0.40\textwidth]{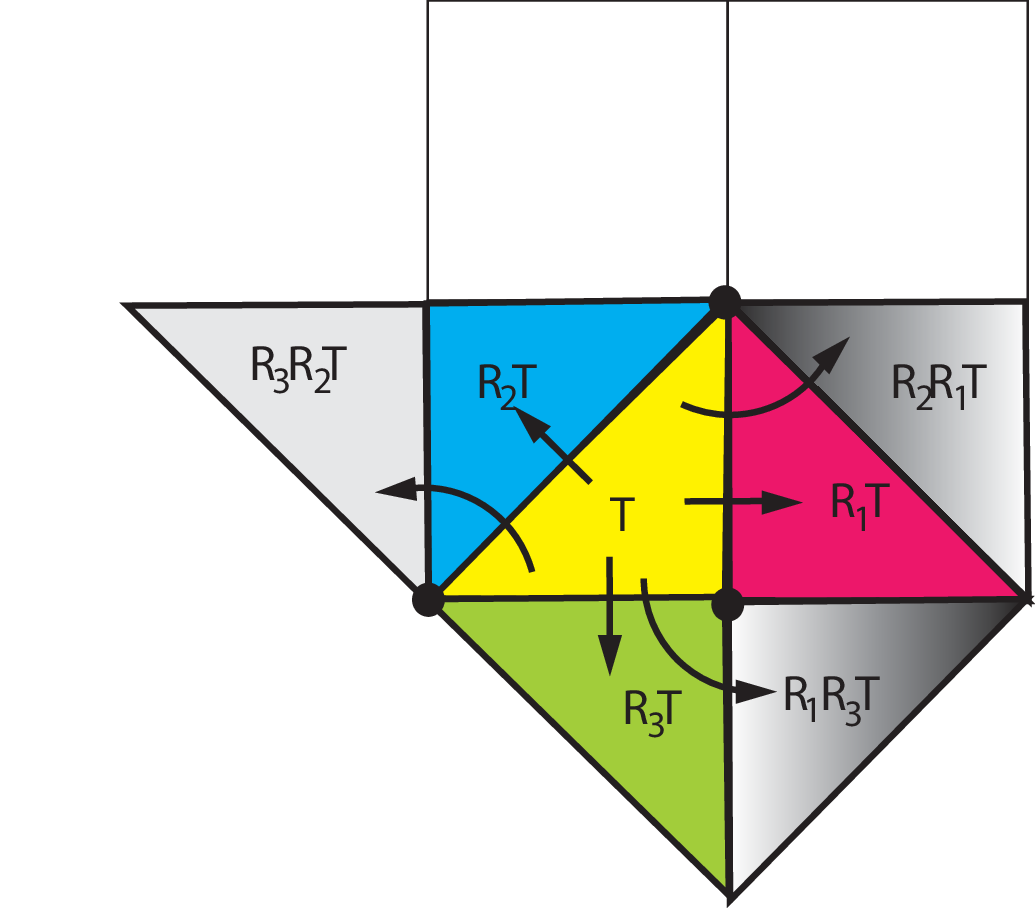}
 \caption{triangle motions for a square}\label{fig:square}
\end{floatingfigure}

The surface $S_g$ is covered by applying all possible sequences of reflections $R_i$'s to the triangle $T$ or by applying all possible sequences of rotations $\Rc=R_1R_3$, $\Sc=R_3R_2$ and $\Tc=R_2R_1$ (around the triangles  vertices) to the fundamental polygon\footnote{This assumes that the Tiling is already given and therefore all the vertices of the triangulation.}.

The group generated by $\{\Rc, \Sc, \Tc\}$ is an important subgroup of $\Delta_g(h,k,l)$ which has index $2$. It is called a von Dyck group and can be characterized as follows $\Delta^+_g(h,k,l)=\{\Rc, \Sc, \Tc, \text{ such that } \Rc^h=\Sc^k=\Tc^l=\Rc\Sc\Tc=\I$,  plus other relations depending on the surface $S_g\}$.

There is more structure to this: the fundamental triangle  $T$ associated with the triangle group $\Delta_g(h, k, l)$ has angles $\pi/h, \pi/k, \pi/l$. 

Spherical triangles satisfy $1/h+1/k+1/l>1$ and only a finite number of configurations\footnote{These configurations are related to the regular polyhedra [\ref{parala:polygon}].} are possible. Triangles in the plane or Torus satisfy  $1/h+1/k+1/l=1$ and  also a finite number of cases are possible. 
However, the most interesting cases correspond to hyperbolic triangles: those for which the inequality  $1/h+1/k+1/l<1$ holds.

The triangle groups related to the problem $\{F,$ $p, q\}$ are those of the form $\Delta_g(2,q,p)$\footnote{This is because when doing barycentric subdividion of a regular polygon we get right triangles and so $h=2$.}. Their corresponding von Dyck groups (we write them as $\Delta^+_g(2,q,p)$) can be  presented as generated by  $\Rc$ and $\Sc,$ one of them of order $2$. Indeed, writing $\Tc=\Sc^{-1}\Rc$, we can present the von Dyck  group as follows: $\Delta^+_g(2,q,p)=\{\Rc, \Sc, \text{ such that } \Rc^2=\Sc^p=(\Sc^{-1}\Rc)^q=\I,$ plus relations depending on the surface $S_g\}$.

We will be mainly concerned with orientable\footnote{Not those like the M\"obius band or the projective plane.} closed surfaces (that is without boundary). These surfaces are classified by their genus $g$ (an integer $g\geq0$). This is the number of holes they have or the number of handles attached to a sphere. For instance the genus of a sphere is $0$ and that of a torus is $1$. All the surfaces of genus $g>1$ fall into the Hyperbolic realm. Figure~\ref{fig:g4}  pictures a surface of genus $4$. If a surface $S_g$ of genus $g$ has a Tiling by convex\footnote{We exclude polygons like Star polygons because they may be misleading when computing the Euler-Poincar\'e characteristic. The safest way to compute the Euler-Poincar\'e characteristic is to use a triangulation on $S_g$.} polygons and $V=$number of vertices, $E=$number of edges, $F=$number of faces for this Tiling,  then the Euler-Poincar\'e Characteristic: $\chi(S_g):=V-E+F$ is an integer number which does not depend on the Tiling considered on $S_g$.  We have the following.
\begin{equation}
\chi(S_g)=V-E+F=2-2 g
\end{equation}

If a surface $S_g$ of genus $g>1$ has a Tiling by regular $p-$gons with incidence $q$ at each polygon vertex, then there is an induced barycentric Tessellation by hyperbolic right triangles with angles $\pi/p, \pi/q$. These triangles are in correspondence with the elements of the group $\Delta_g(2,p,q)$ and its order can be computed by the formula \citep[see][]{kawe96} 
\begin{align}
|\Delta_g(2,p,q)|&=\frac{\text{ (Hyperbolic Area of }S_g)}{\text{(Hyperbolic Area of the }(2, p, q)-\text{triangle})} =\\
&=\frac{-2 \pi \chi(S_g)}{(\frac{\pi}{2}-\frac{\pi}{p}-\frac{\pi}{q})}=\frac{8 p q (g-1)}{(p-2)(q-2)-4}\label{eq:chi}
\end{align}

\begin{remark}\label{rem:first} Formula (11) holds also for the sphere (just put $g=0$).
\end{remark}

\begin{remark}\label{rem:second}
If on a surface $S_g$ of genus $g$ we can solve the problem $\{F, p, q\}$, then by taking apart the polygons and counting edges we get $q V=2 E=p F$. Thus, $\chi(S_g)=2 E (\frac{1}{p}+\frac{1}{q}-\frac{1}{2})$. So, another way of writing Formula (11) is the following $|\Delta_g(2,p,q)|=4 E= 2 p F = 2 q V$.
\end{remark}

\begin{figure}
 \vspace{3mm}
 \centering
    \includegraphics[width=0.70\textwidth]{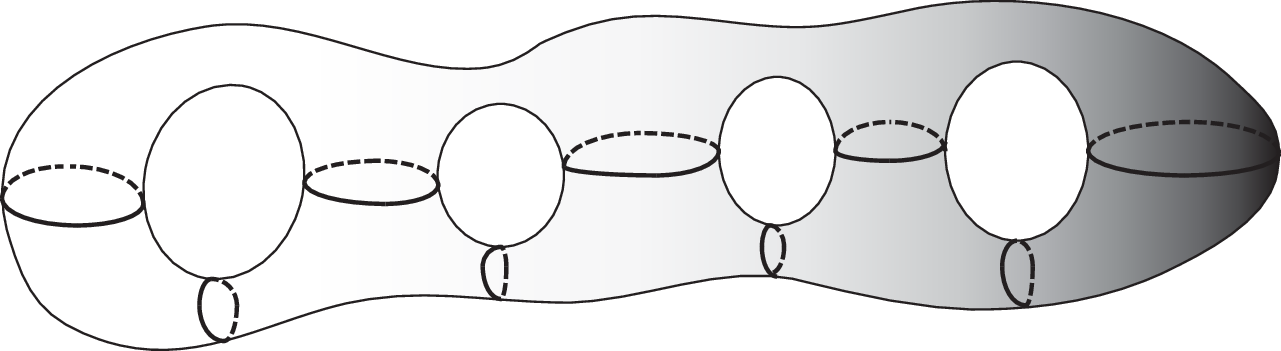}
 \caption{a genus $g=4$ surface}\label{fig:g4}
\end{figure}

\begin{remark}\label{rem:three} As long as we know the number of faces (edges or vertices) of a regular Tessellation of type $\{p, q\}$ on a surface $S_g$ we can compute its genus: 
\begin{equation}\label{eq:genusformula}
g=1-\frac{\chi(S_g)}{2}=1+\frac{1}{2}(\frac{1}{2}-\frac{1}{p}-\frac{1}{q}) p F.
\end{equation}
\end{remark}

\begin{remark}There is still another formula which allows us to compute the Euler-Poincar\'e characteristic in case we know the number of faces $F$ and vertices $V$ of a regular $\{p, q\}$ Tiling on $S_g$:
\begin{equation}
(p-2)(q-2)=4\left(1-\frac{\chi(S_g)}{V}\right)\left(1-\frac{\chi(S_g)}{F}\right)
\end{equation} 
\end{remark}

\begin{remark} The full triangle group $\Delta(2,p, q)$\footnote{This is Definition 4 without the relations for the surface $S_g$. } takes a fundamental right triangle with angles $\pi/p, \pi/q$ in the Poincar\'e Disk and moves it to cover the whole Disk giving a pattern that also contains p-gons where q polygons meet at each vertex\footnote{As shown in the picture \href{http://en.wikipedia.org/wiki/File:Order-3_heptakis_heptagonal_tiling.png}{\tt (2,3,7)-triangle Tiling}.}. One obtains such patterns as pictured in Escher Limit Circle drawings\footnote{See for instance \href{http://www.josleys.com/show_gallery.php?galid=325}{\tt http://www.josleys.com/show\_gallery.php?galid=325}.}. This group is infinite and for each surface $S_g$ containing a regular $\{p, q\}$ Tiling, the group $\Delta_g(2,p, q)$ is the quotient of $\Delta(2,p, q)$ by a subgroup containing at least the dihedral group preserving the p-gon.
\end{remark}

{\centering{
\section{Group structure and local incidence}\label{sec:incidence}
}}

Consider now  a pitch segment $\langle x_1, x_2, \dots, x_n\rangle$, $x_i\in \Zt$ for $0\leq i\leq n$ and assume that we allow the contextual transformations generated by $\{p_{12}, p_{23}, \dots, p_{i,i+1}, \dots, p_{n1}\}$. The group generated by these inversions is denoted $\widetilde G_n$. This group is contained in a bigger group $G_n$ generated by all the inversions $p_{ij}$'s. As in [\ref{parala:g4}] $\widetilde G_n$ is mapped into the group generated by $\{p_{12}, \sigma\}$ and later on we will show that this is a von Dyck group.

We describe the structure and basis of $G_n$ and $\widetilde G_n$.

\begin{proposition}\label{prop:strut}
The group of contextual inversions $G_n$ is isomorphic to the semidirect product of $\Z_2$ and the abelian group  $\Zt^{n-1}$. Bases for this last abelian group are given by
$$\{T^2_1, T^3_1, \dots, T^n_1\} \text{ or } \{T^2_1, T^3_2, \dots, T^n_2\}.$$

The group of contextual inversions $\widetilde G_n$ is isomorphic to $\Z_2\ltimes\Zt^{n-2}$ if $n$ is even and to $\Z_2\ltimes\Zt^{n-1}$ in case $n$ is odd. Bases for the abelian parts in these two cases are:
$$\{T^3_1, T^4_2, T^5_1, \dots, T^n_2\}\quad\text{ if }n\text{ is even,}$$
$$\{T^2_1, T^3_1, T^4_2, \dots, T^{n-1}_2, T^n_1\}\quad\text{ if }n\text{ is odd.}$$
\end{proposition} 

\begin{proof} $G_n$ is generated by $\{p_{12}, p_{13}, \dots, p_{1n}, p_{23}, p_{24}, \dots, p_{2n}, p_{34}, \dots, p_{3n}, \dots, p_{n-1,n}\}$ and also by $p_{12}$ and the abelian generators $\{T^3_2, \dots, T^n_2, T^3_1, \dots, T^n_1, T^{34}_{12}, \dots, T^{3n}_{12}, \dots, T^{n-1,n}_{12}\}$. Since we have in additive notation $T^{hk}_{ij}=T^h_i+T^k_j$ (see Lemma~\ref{lem:second}), the set $\{T^3_2, \dots, T^n_1, T^3_1, \dots, T^n_1\}$ already generates the abelian piece. Now, since $T^3_1-T^3_2=\dots=T^n_1-T^n_2=T^2_1$, we are left with the generators $\{T^2_1, T^3_1, \dots, T^n_1\}$ which are readily seen to be generically independent. The remaining conclusions about $G_n$ follow from Lemma~\ref{lem:second}. 

$\widetilde G_n$ is generated by $\{p_{12}, p_{23}, p_{34}, \dots, p_{n-1,n}, p_{1n}\}$. So, $p_{12}$ and the abelian piece generated by $\{T^3_1, T^{34}_{12}, \dots, T^{n-1,n}_{12}, T^n_2\}$ define the whole group. Suppose that $n=2 k$ and write these generators as $\{T^3_1, T^3_1+T^4_2, T^4_2+T^5_1, \dots, T^{2 k-1}_1+T^{2 k}_2, T^{2 k}_2\}$. Then, an equivalent set of generators is $\{T^3_1,T^4_2, T^5_1, T^6_2, \dots, T^{2k-1}_1, T^{2k}_2\}$, and these are $n-2$ independent generators. Therefore, $\widetilde G_n\cong\Z_2\ltimes\Zt^{n-2}$.

If $n=2k+1$ the generators of the abelian piece are $\{T^3_1, T^3_1+T^4_2, T^4_2+T^5_1, \dots, T^{2 k-1}_1+T^{2 k}_2, T^{2k}_2+T^{2k+1}_1, T^{2 k+1}_2\}$. We can replace this set with $\{T^3_1, T^4_2, \dots, T^{2k}_2, T^{2k+1}_1, T^{2k+1}_2\}$ or the equivalent set $\{T^2_1, T^3_1, T^4_2, \dots, T^{2k}_2, T^{2k+1}_1\}$. These generators are equivalent to any of the given bases for the abelian part of $G_{2 k+1}$. So, in this case $\widetilde G_{2k+1}=G_{2k+1}$.
\end{proof}

What is the image of a single pitch segment class $\langle x_1, x_2, \dots, x_n\rangle$ under $\widetilde G_n$? Do we recover the whole set class $(T/I)\bullet$$\langle x_1, x_2, \dots, x_n\rangle$? Obviously $\widetilde G_n\bullet$$\langle x_1, x_2, \dots, x_n\rangle$ is contained in $(T/I)\bullet$$\langle x_1, x_2, \dots, x_n\rangle$ because $p_{ij}$'s are made up of translations $T_m$'s and inversion $I$.

If $n$ is odd a condition for $\widetilde G_n\bullet$$\langle x_1, x_2, \dots, x_n\rangle$ to equal $(T/I)\bullet$$\langle x_1, x_2, \dots, x_n\rangle$ is the following:

\begin{condition}\label{cond:first} There exist integer numbers $\lambda_2, \lambda_3, \dots, \lambda_n,$ such that 
$$\lambda_2 (x_2-x_1)+\lambda_3 (x_3-x_1)+\dots+\lambda_n (x_n-x_1)\equiv1\text{ mod  }12.$$ 
\end{condition} 

This makes sure that  by applying $T_1\in\widetilde G_n$. We use the equality deduced  in Proposition 1 $\widetilde G_n=G_n$ and the first basis of $G_n$. The operator $Op=(T^2_1)^{\lambda_2}\circ(T^3_1)^{\lambda_3}\circ\dots\circ(T^n_1)^{\lambda_n}$ is translation by $1$ when applied to the segment $\langle x_1, x_2, \dots, x_n\rangle$. Now, the operator $p_{12}\circ(Op)^{(-x_1-x_2)}$ applies like $I$ on the segment.

In case $n$ is even we use 

\begin{condition}There are integer numbers $\lambda_3, \lambda_4, \dots, \lambda_n,$ such that 
$$\lambda_3 (x_3-x_1)+\lambda_4 (x_4-x_2)+\lambda_5 (x_5-x_1)+\dots+\lambda_{n-1} (x_{n-1}-x_1)+\lambda_n (x_n-x_2)\equiv1\text{ mod  }12.$$ 
\end{condition}

The proof that this suffices is similar to that of Condition 1. We observe that since  there are many bases for $\widetilde G_n$,  other conditions are possible to obtain $\widetilde G_n\bullet$$\langle x_1, x_2, \dots, x_n\rangle$$=$ $(T/I)\bullet$ $\langle x_1, x_2, \dots, x_n\rangle$.

As a counterexample, the pitch segment $\langle0,2,4,6,8\rangle$ does not satisfy Condition 1. Either we treat it as a pathological case, or we deal with it as a particular case, aside from the considerations we are explaining in this article.

\subsection{The gluing procedure}\label{parala:procedure}
We map $\widetilde G_n$ into the group generated by $\{p_{12}, \sigma\}$ as follows (using Formula~\eqref{eq:second}): 

\begin{floatingfigure}[r]{0.27\textwidth}
 \centering
    \includegraphics[width=0.24\textwidth]{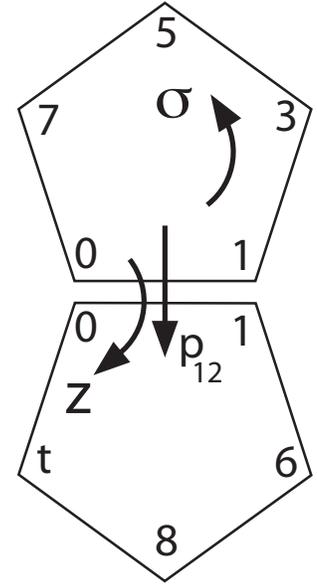}
 \caption{matching a pentagon}\label{fig:pentagon}
 \vspace{3mm}
\end{floatingfigure}

$p_{23}=\sigma^{-1}p_{12}\sigma,$  $p_{34}=\sigma^{-2}p_{12}\sigma^2, \dots, p_{1n}=\sigma^{-(n-1)}p_{12}\sigma^{n-1}.$ As we did in [\ref{parala:g3}] for triangles,  we consider the regular polygon of $n$ sides $P_n$ associated to the segment $\langle x_1, x_2, \dots, x_n\rangle$, and glue it to the polygon associated with $\sigma p_{12}\langle x_1, x_2, \dots, x_n\rangle$$=$ $\langle x_1,$ $I_{x_1+x_2}(x_3), \dots, I_{x_1+x_2}(x_n), x_2\rangle$ along the side $\langle x_1, x_2\rangle$ (Figure~\ref{fig:pentagon}). Namely, the element $z=\sigma p_{12}$ is a rotation around the vertex $x_1$ that takes one polygon into the next and $\sigma^{-1}$ is rotation of the polygon vertices.

The order of $z$ is the incidence at each vertex. For instance,  by successively rotating with $z$   the polygon of Figure~\ref{fig:pentagon} we get back the pitch $\langle0,1,3,5,7\rangle$ after $10$ times: 
$\langle0,1,3,5,7\rangle\xrightarrow{z}$ $\langle0,t,8,$ $6,1\rangle\xrightarrow{z}\langle0,2,4,9,t\rangle\xrightarrow{z}\langle0,t,5,4,2\rangle\xrightarrow{z}\langle0,5,6,8,t\rangle\xrightarrow{z}\langle0,e,9,7,5\rangle$ $\xrightarrow{z}\langle0,2,4,6,e\rangle\xrightarrow{z}\langle0,t,8,3,2\rangle\xrightarrow{z}\langle0,2,7,8,t\rangle\xrightarrow{z}\langle0,7,6,4,2\rangle\xrightarrow{z}\langle0,1,3,5,7\rangle$. 

We can picture this in the Poincar\'e Disk with $10$ meeting pentagons (see Figure~\ref{fig:ten}).

Transporting these prescriptions along all sides and vertices we  get a surface which is regularly tessellated by equal polygons of $n$ sides. If the Tiling is finite the surface is closed since each edge belongs to two and only two faces (there are no free edges belonging to only one polygon) and each vertex has the same number of incident edges (this is the order of $z$). Thus the group generated by $\{p_{12}, \sigma^{-1}\}$ is a von Dyck group of a regular Tessellation  on a surface of genus $g$ since it can be described by generators and relations as $\{ p_{12}, \sigma^{-1}, \text{ such that } p_{12}^2=\sigma^n=(\sigma p_{12})^q=\I$ plus relations coming from $S_g\}$.

\begin{proposition} Assume the pitch segment $\langle x_1, x_2, \dots, x_n\rangle$ induces a Regular Tessellation by regular polygons of $n$ sides on a closed surface of genus $g$. Then, the incidence of the polygons at each vertex is $n$  for even $n$ and $2 n$ if $n$ is odd. Therefore,  the group generated by $p_{12}$ and $\sigma$ is isomorphic to $\Delta^+_g(2,n,n)$ if $n$ is even and $\Delta^+_g(2,n,2 n)$ if $n$ is odd.
\end{proposition}

\begin{proof} It is enough to show that $q=n$ if $n$ is even and $q=2 n$ if n is odd. We have  that $p_{12}p_{23}\dots p_{1n}$$=$$p_{12}\sigma^{-1}p_{12}\sigma^{-1}\dots p_{12}\sigma^{-1}=(\sigma p_{12})^{-n}=z^{-n}$. On the other hand, $p_{12}p_{23}p_{34}p_{45}\dots$ $p_{n-1,n}p_{1n}$$=$$T^3_1p_{12}T^{34}_{12}p_{12}T^{45}_{12}\dots p_{12}T^{n-1,n}_{12}p_{12}T^n_2$. If $n=2 k$ we write $z^{-n}=T^3_1p_{12}T^{34}_{12}p_{12}T^{45}_{12}\dots$ $p_{12}T^{2k-1,2k}_{12}p_{12}T^{2k}_2=T^3_1T^{12}_{34}T^{45}_{12}\dots T^{12}_{2k-1,2k}T^{2k}_2=T^3_1T^1_3T^2_4T^4_2T^5_1\dots T^1_{2k-1}T^2_{2k}T^{2k}_2=\I$. This shows $z^n=\I$ for $n$ even. If $n=2k+1$, $z^{-n}=T^3_1T^{12}_{34}T^{45}_{12}\dots T^{12}_{2k-1,2k}T^{2k,2k+1}_{12}p_{1,2k+1}=T^{2k+1}_1p_{1,2k+1}$. However, by Lemma~\ref{lem:second} $z^{-2n}=T^{2k+1}_1p_{1,2k+1}T^{2k+1}_1p_{1,2k+1}=T^{2k+1}_1T^1_{2k+1}=\I$. This shows the proposition.
\end{proof}

\subsection{Mapping $\widetilde G_n$}\label{parala:normal}
Let us call by $\Nc$ the image of $\widetilde G_n$ into the von Dyck group. 
Then it is readily seen that this group is normal. Indeed, $p_{12}\Nc p_{12}\subseteq \Nc$ and $\sigma \Nc\sigma^{-1}\subseteq \Nc$. We just check this on the generating elements of $\Nc$. Therefore we can make sense of the quotient group $\Delta^+_g(2,n,n)/\Nc$ ($n$ even) and the quotient group $\Delta^+_g(2,n,$ $2n)/\Nc$ ($n$ odd). These quotients are isomorphic to $\Z_m=\langle\widehat \sigma$ plus relations induced by $S_g\rangle$ where $m$ divides $n$. Namely, in both cases $\Nc$ is a normal subgroup of index $m$ in a von Dyck group. On the other hand we have the exact sequence of groups $1\rightarrow Ab\rightarrow \widetilde G_n \rightarrow \Nc \rightarrow 1$, where $Ab$ is an abelian subgroup of $\widetilde G_n$.  That is $Ab$ is a subgroup of $\Zt^{n-2}$ for $n$ even and a subgroup of $\Zt^{n-1}$ for $n$ odd. 

We give an example to see how these groups look like.

\begin{example}\label{exa:four} Consider the pitch segment $\langle0,4,7, t, 2\rangle$ and use the basis of Proposition 1. 
\begin{floatingfigure}[r]{0.53\textwidth}
 \centering
    \includegraphics[width=0.50\textwidth]{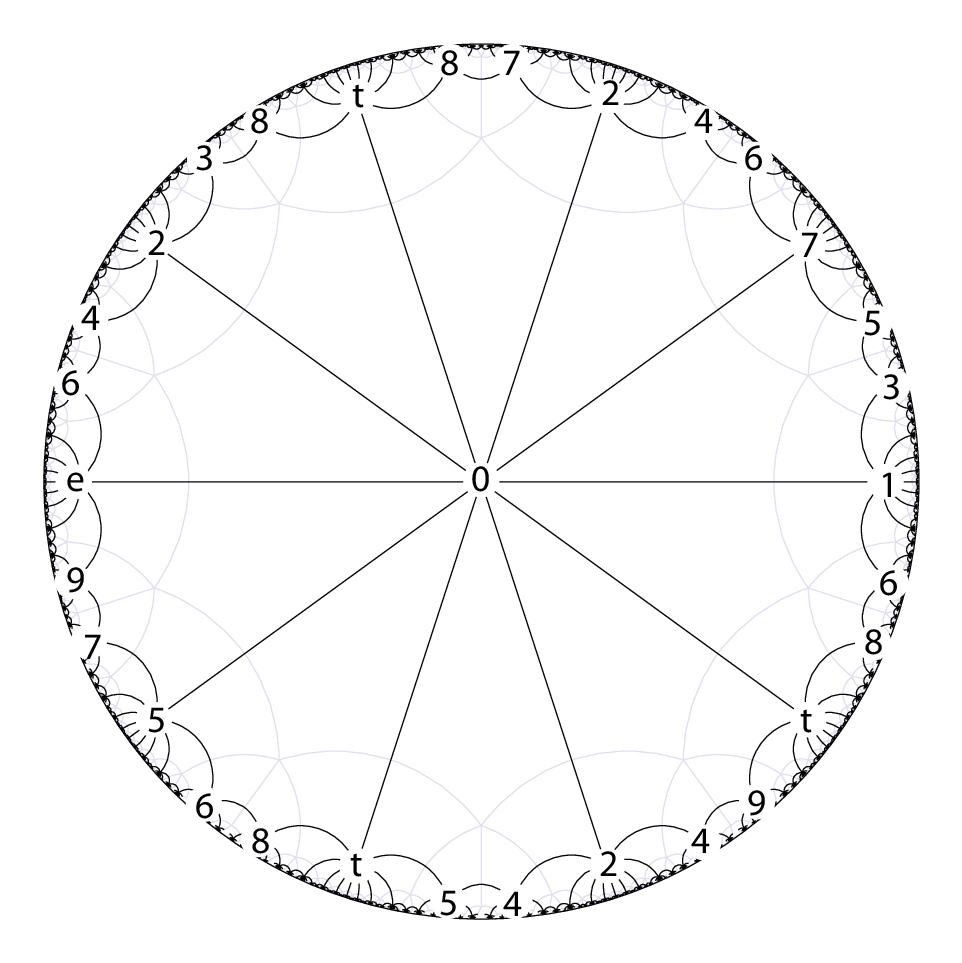}
 \caption{ten pentagons incidence}\label{fig:ten}
\end{floatingfigure}
We see that the abelian part of $\widetilde G_5$ is generated by the translations in four different directions $T^2_1\equiv T_4$, $T^3_1\equiv T_7$, $T^4_1\equiv T_{10}$, and $T^5_1\equiv T_2$. In this case, any element of $\widetilde G_5$ is written in the form $p_{12}T_4^{\lambda_1}T_7^{\lambda_2}T_{10}^{\lambda_3}T_2^{\lambda_4}$, where the $\lambda_i$'s are integer exponents.  Condition 1 is readily checked and therefore the image $\Nc$ of $\widetilde G_5$ is isomorphic to the dihedral group $\D_{12}$. Indeed, the operator $T^j_i$ is written as $T^j_i=\I+C_j-C_i$ where $C_i$ is a $5\times 5$ matrix with zero entries everywhere except in he $i$-th column  where all entries are $1$. One checks that $(T^j_i)^{\lambda}=\I+\lambda (C_j-C_i)$ and these matrices go down to $\I + \lambda T_{x_j-x_i}$ as they are restricted to the segment $(T/I)$-orbit. $\D_{12}$ is the group generated by two letters $s,t$ with relations $s^{12}=t^2=1,$ $tst=s^{-1}$, and the isomorphism with $\Nc$ is produced by sending $p_{12}\rightarrow t$, $T_1\rightarrow s$. Since $\Nc\simeq\D_{12}$ and has index $5$ into $\Delta_g^+(2,5,10)$, then by Formula~\eqref{eq:chi} one gets $24\times5\times2=|\Delta_g(2,5,10)|=20 (g-1)$. Thus, a surface containing the group generated by $p_{12}$ and $\sigma$ must have genus $13$. However this is not the surface containing an action of the whole group $\widetilde G_5$ acting on the orbit. For this we have to compute the kernel of the morphism $\varphi:\Z^4\rightarrow \Zt$ defined as $\varphi(\lambda_1,\lambda_2,\lambda_3,\lambda_4)=4\lambda_1+7\lambda_2+10\lambda_3+2\lambda_4$ so that we get the isomorphism $Ab\simeq \Z^4/Ker \varphi$. To determine $Ker \varphi$ we notice that it is generated by the vectors $\{e_1=(0,10,0,1), e_2=(0,2,1,0), e_3=(1,8,0,0), e_4=(0,12,0,0), e_5=(3,0,0,0), e_6=(0,0,6,0),e_7=(0,0,0,6)\}$. Also $\{e_1,e_2,e_3,e_4\}$ generates $Ker \varphi$ and is an independent basis for it. Any element in $\Z^4$ can be written in terms of this basis of $Ker \varphi$ as follows $(\mu_1,\mu_2,\mu_3,\mu_4)= \mu_1 e_3+\mu_3 e_2 +\mu_4 e_1 + \frac{1}{12}(\mu_2-8\mu_1-2\mu_3-10\mu_4) e_5$. Clearly, we obtain $Ab\simeq \Z^4/Ker \varphi\simeq\Zt$. 
\end{example} 

Consider a genus $g$ closed surface $S_g$ on which we have a Tessellation of type $\{n, n\}$ for $n$ even or $\{n, 2n\}$ if $n$ is odd. Moreover, let us assume the set of faces  contain at least the $24$ faces associated with the $(T/I)-$forms of a given pitch class segment $\langle x_1, x_2, \dots x_n\rangle$. That is, $F=24N$. What would the Euler-Poincar\'e characteristic and the minimal possible genus of $S_g$ be? This problem depends on the possibility of embedding a Graph on a surface of genus $g$. A topological connected Graph (i.e. points joined by arc segments) can always be embedded into a surface of some genus $g$. Even into a non-orientable surface. However determining the minimum embedding genus $g$ is a very difficult problem that belongs to Topological Graph Theory \citep[see][chap. 7]{han04}, \citep[also][]{grotu01}.

Putting $p=q=n$ and $F=24N$ in Formula~\eqref{eq:genusformula} we get $g=1+6 (n-4) N$ for $n$ even. In the $n$ odd case $p=n, q=2n$ and we get $g=1+6 (n-3) N$. 

In Table~\ref{tab:incidence} we shows for small $n$ the Tessellation type on $S_g$ for the corresponding pitch class segment and what could the minimal genus be if a Tessellation of that type is found on $S_g$. Warning: the genera shown in the table only reflect the formulas obtained. This does not mean that you would get a Tiling representing the pitch set class $sc(x_1x_2\dots x_n)$ with the minimal genus displayed in the table. The only certain values for this hold in genus $1$.

\begin{table}[th]
\begin{center}
\caption{Incidence and possible genera for small pitch class segments}\label{tab:incidence}
\vspace{3mm}
\begin{tabular}{|c|c|c|c|c|} 
\hline  $n$ & pitch segment & Tiling type & genus formula &  possible values\\

\hline   3 & $\langle x_1, x_2, x_3\rangle$ & $\{3,6\}$ & $g=1$ & 1\\
   4 & $\langle x_1, x_2, x_3, x_4\rangle$ & $\{4,4\}$ & $g=1$ & 1\\
   5 & $\langle x_1, x_2, x_3, x_4, x_5\rangle$ & $\{5,10\}$ & $g=1+12 N$ & 13, 25, 37, 49, 61\\

   6 & $\langle x_1, x_2, x_3, x_4, x_5, x_6\rangle$ & $\{6,6\}$ & $g=1+12 N$ &13, 25, 37, 49, 61\\

7 & $\langle x_1, x_2, x_3, x_4, x_5, x_6, x_7\rangle$ & $\{7,14\}$ & $g=1+24 N$ & 25,49, 73, \dots\\

8 & $\langle x_1, x_2, x_3, x_4, x_5, x_6, x_7, x_8\rangle$ & $\{8,8\}$ & $g=1+24 N$ & 25,49, 73, \dots\\

9 & $\langle x_1, x_2, x_3, x_4, x_5, x_6, x_7, x_8, x_9\rangle$ & $\{9,18\}$ & $g=1+36 N$ & 37, 73, \dots\\

10 & $\langle x_1, x_2, x_3, x_4, x_5, x_6, x_7, x_8, x_9, x_{10}\rangle$ &$\{10,10\}$ & $g=1+36 N$ & 37, 73, \dots\\

11 & $\langle x_1, x_2, x_3, x_4, x_5, x_6, x_7, x_8, x_9, x_{10}, x_{11}\rangle$ &$\{11,22\}$ & $g=1+48 N$ & 49, \dots\\

12 & $\langle x_1, x_2, x_3, x_4, x_5, x_6, x_7, x_8, x_9, x_{10}, x_{11}, x_{12}\rangle$ & $\{12,12\}$ & $g=1+48 N$ &49, \dots\\

\hline
\end{tabular}
\end{center}
\end{table}

\begin{remark} A generic pitch class segment holds a very high genus. Indeed, suppose we can view the group of inversions $\widetilde G_n$ as a group of automorphisms on a surface of genus $g$, then by a famous theorem of Hurwitz, the order of $\widetilde G_n$ is bounded above by $84(g-1)$. Thus, in the case of $n$ odd Proposition~\ref{prop:strut} would tell us that $g\geq1+\frac{|\Z_2\ltimes \Zt^{n-1}|}{84}=1+\frac{2\times12^{n-1}}{84}$. This is $g\geq495$ if $n=5$. However, as seen in Example~\ref{exa:four}$,\widetilde G_5$ cuts down to a group of order $2\times12^2$ for the pitch segment $\langle0,4,7,t,2\rangle$, i.e. $g\geq1+\frac{2\times12^2}{84}\geq5$. Latter we will see that this example holds a much higher genus.
\end{remark}

{\centering{
\section{Minimal model surface and pentachords}\label{sec:surface}
}}

\subsection{A Tiling containing 24 pentagons}\label{parala:icosa}
In a similar way as did in Example~\ref{exa:four} we will work with a fixed pitch segment $\langle x_1,x_2,x_3,$ $x_4,x_5\rangle$ and its $(T/I)$ orbit. We assume that this pitch segment satisfies Condition~\ref{cond:first}. 

\begin{floatingfigure}[r]{0.43\textwidth}
 \centering
    \includegraphics[width=0.40\textwidth]{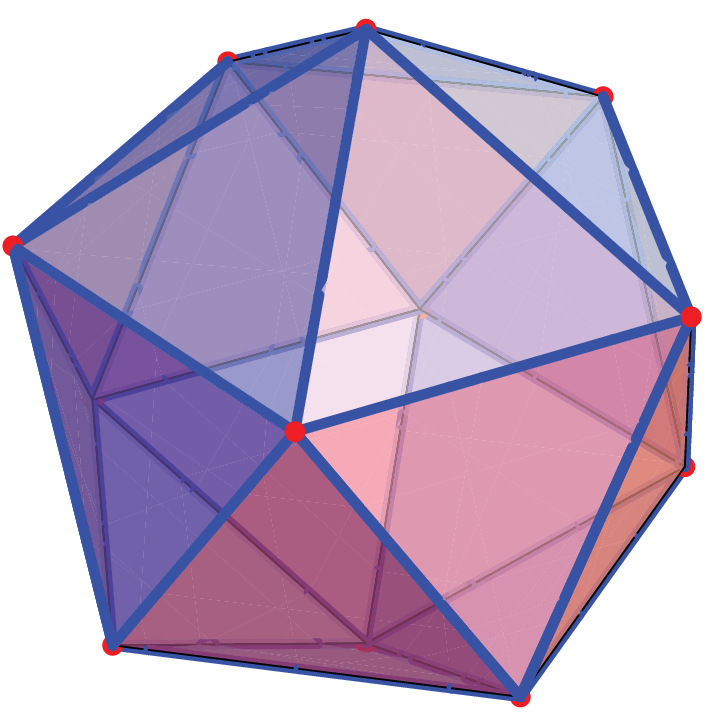}
 \caption
 {icosahedron}\label{fig:icosahedron}
\end{floatingfigure}

Then the image $\Nc$ of $\widetilde G_5$ into the group generated by  $\{p_{12}, \sigma\}$ is the dihedral group $\D_{12}$. If a Tiling by pentagons with vertex incidence $10$ exists on a genus $g$ surface $S_g$, then the von Dyck group $\Delta^+_g(2,5,10)$ contains $\Nc$ with index $5$. The same computation we did in Example~\ref{exa:four} holds and shows that the genus is $g=13$. 

The problem here is the existence of such a Tiling on a surface of genus $13$. Fortunately, such a Tiling exists and it is related to a surface already studied by Felix Klein \citep{kle84} in connection with the solutions of the quintic equation: the Bring surface\footnote{In the literature is called Bring curve. Here we use the dictionary ``algebraic complete curves over the complex numbers'' $\equiv$ ``closed smooth surfaces with complex structure'' (or Riemann surfaces, in honor of the mathematician Bernhard Riemann).}. The Bring surface is a surface of genus $4$ and can be viewed as a triple branched cover of the icosahedron 
 \citep[see][for such matters]{we05}. The Bring surface possess a Regular Tessellation of type $\{5,5\}$ consisting of twelve hyperbolic pentagons whose centers map to the vertices of the the sphere with icosahedral Tessellation (an inflated icosahedral  balloon framed at its vertices). The $12$ vertices of the pentagons are branched (with branch index $2$) over the same vertices of the Tessellated sphere (each one fifth $(2\pi/5,2\pi/10,2\pi/10)$-hyperbolic triangle of a hyperbolic pentagon maps onto a $(2\pi/5,2\pi/5,2\pi/5)$-spherical triangle of the icosahedral Tessellation of the sphere). This Tessellation was explained in W. Threlfall book \citep{thre32,we05} and can be viewed in Figure~\ref{fig:Bring}\footnote{This is essentially a picture due to Threlfall with some modifications by Weber.} as its image into the Poincar\'e Disk.

The pentagonal tiles $P_i$'s are numbered $1$ to $12$ with repeated tiles in the Figure since we are in the Poincar\'e Disk. The vertices are the red points labeled by $\{\text{a, b, c, d, e, f, g, h, i,}$ $\text{j, k, m}\}$. The fundamental domain of the surface is the union of the $10$ quadrilaterals incident at the center of $P_1$ and forming a regular icosagon (violet dashed boundary) with  identification of sides.

\begin{figure}
 \centering
    \includegraphics[width=\textwidth]{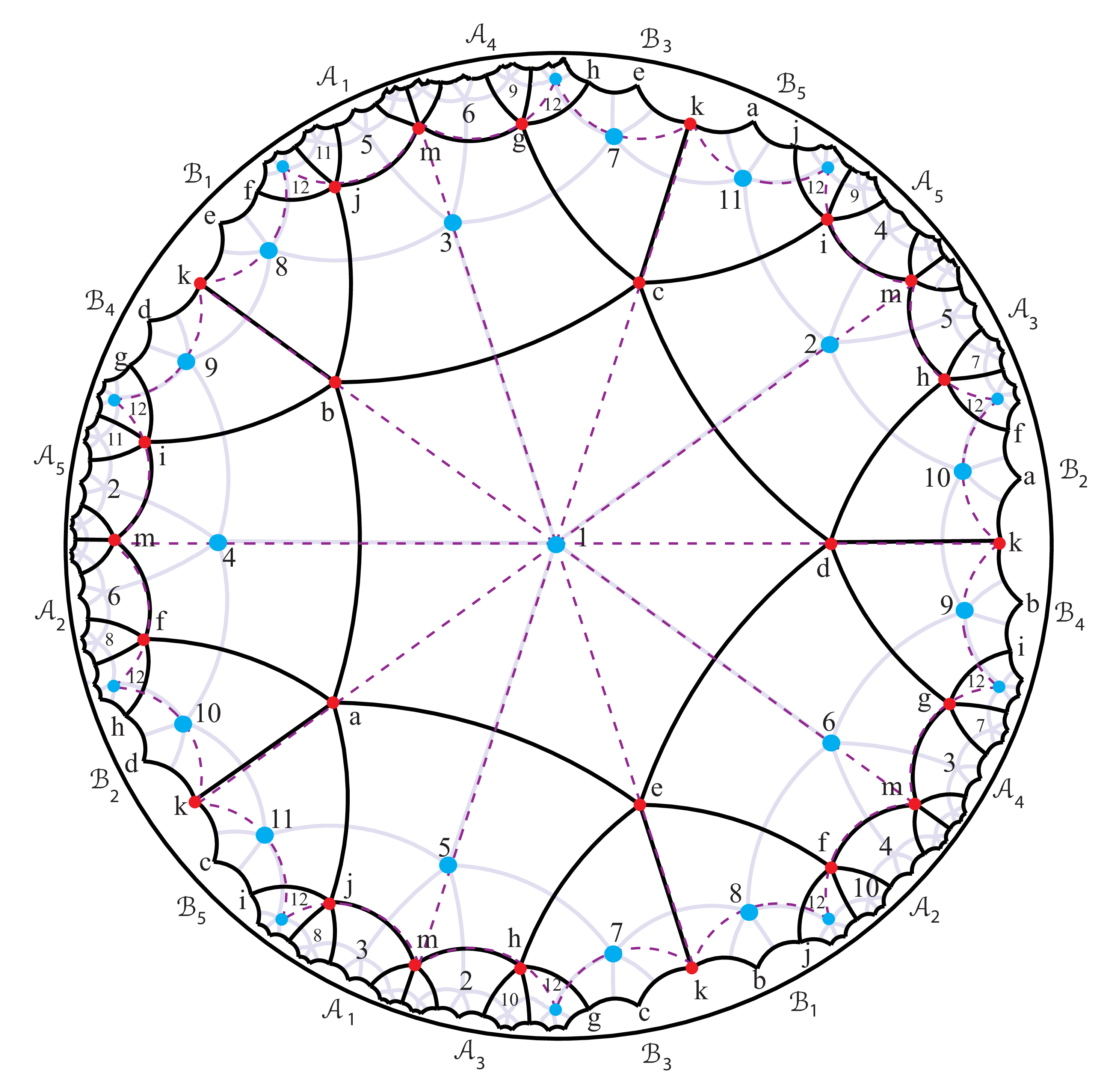}
 \caption
 {Tessellation of Bring surface by 12 pentagons}\label{fig:Bring}
\end{figure}

Cutting with scissors  this icosagon and gluing equivalent sides of its border we get a genus $4$ surface tessellated by $12$ pentagons (some pentagons are recovered by rearranging their triangular pieces). 

Indeed, this $\{5,5\}$ Tessellation has a set of $12$ vertices and $30$ edges, so the Euler-Poincar\'e characteristic is $\chi=12-30+12=-6$, corresponding to genus $4$. The relations satisfied by the sides of the icosagon (labeled $\mathcal{A}_i$ and $\mathcal{B}_i$, $i=1, \dots, 5$) are given in Threlfall's book \citep[see p. 22]{thre32}. They are $\mathcal{A}_1\mathcal{B}_1\mathcal{A}_2\mathcal{B}_2\mathcal{A}_3\mathcal{B}_3\mathcal{A}_4$ $\mathcal{B}_4$$\mathcal{A}_5$$\mathcal{B}_5$ $=\I$, $\mathcal{A}_1\mathcal{A}_4\mathcal{A}_2\mathcal{A}_5\mathcal{A}_3=\I$, and $\mathcal{B}_1\mathcal{B}_3\mathcal{B}_5$$\mathcal{B}_2\mathcal{B}_4=\I$.

Now we proceed to take two copies of this tessellated genus $4$ surface, make cuts along $6$ disjoint edges in each sheet and glue the sheets along each pair of edges (lips) produced in the cuts. We get in this way a two sheeted cover $\Gamma$ of the Bring surface $B$ which is branched over the $12$ points $\{\text{a, b, c, d, e, f, g, h, i,}$ $\text{j, k, m}\}$. One way of choosing the cuts is the following: edge(a,k)=$P_{10}\cap P_{11}$, edge(b,c)=$P_{1}\cap P_{3}$,  edge(d,e)=$P_{1}\cap P_{6}$, edge(f,j)=$P_{8}\cap P_{12}$, edge(h,m)=$P_{2}\cap P_{5}$, edge(i,g)=$P_{9}\cap P_{12}$. The new surface $\Gamma$ has an induced Tessellation consisting of $24$ hyperbolic pentagons, $12$ vertices and $60$ edges, but at each vertex $10$ pentagons meet. Indeed, as we turn around a vertex we cover $5$ pentagons; by crossing the cut we go into the other sheet and with another turn we cover the remaining $5$ pentagons that meet around the chosen vertex. Therefore, $\Gamma$ carries a Tiling of type $\{5,10\}$ and has invariants $\chi(\Gamma)=12-60+24=-24$, and $g=13$. 

On the other hand, Hurwitz formula gives the same result: $\chi(\Gamma)=2.\chi(B)- \text{ramification}$ $\text{index}=2.(-6)-12.(2-1)=-24$.

From what it was said in this paragraph we have the following:

\begin{proposition}\label{prop:three} There is a surface $\Gamma$ of genus $13$ which carries a Tiling by $24$ hyperbolic pentagons of type $\{5,10\}$. This surface is a $2$-cover of the Bring surface $B$ branched over the $12$ vertices of the $\{5,5\}$ Tessellation of $B$. The von Dyck group $\Delta^+_{13}(2,5,10)$ exists and has order $120$. 
\end{proposition}

\subsection{Construction of the surface model}\label{parala:construction}
We cannot claim that the genus $13$ surface $\Gamma$ is the model we are looking for. Indeed, by starting with any pentachord pitch class segment from a given tile in $\Gamma$ and walking trough all tiles by the inversions in $\widetilde G_5$ we find inconsistencies. That is clear because as said in Example~\ref{exa:four} the surface $\Gamma$ is not acted  by the whole group $\widetilde G_5$.

Start with a fundamental tile $T$: this is a pentagon (hyperbolic) with labeled vertices, and the labels are numbers within the (chromatic) pitch set $\{0,1,2,3,4,5,6,7,8,9,t,e\}$\footnote{Actually, any pitch set with abelian group operations on it could be used for this purpose, e.g. a diatonic group $(\Z_7,+)$.}. The labels of $T$ correspond to a given ordered pitch segment $\langle x_1,x_2,x_3,x_4,x_5\rangle$. The group $\widetilde G_5$ has the inversions $\mathscr{G}=\{p_{12},p_{23},p_{34},p_{45},p_{15}\}$ as generator set and relations contained at least in $\mathscr{R}=\{p_{12}^2=p_{23}^2=p_{34}^2=p_{45}^2=p_{15}^2=(p_{12}p_{23}p_{34}p_{45}p_{15})^2=\I, \text{ commuting relations among the}$ $\text{products } p_{ij}p_{kl}\}$.

A walk in $\mathscr{G}$ is a finite sequence $w=(p_1,p_2,\dots,p_n)$ where $p_i\in \mathscr{G}$ for $i=1,\dots,n$. The set  $\mathscr{W}$ of all walks\footnote{In other contexts these are called words.} has a group structure by concatenation of walks: $(q_1,q_2,\dots,q_m)\circ(p_1,p_2,\dots,p_n)=(p_1,p_2,\dots,p_n,q_1,q_2,\dots,q_m)$; the identity of $\mathscr{W}$ is the empty walk $(\ )$\footnote{This means do not move.}. 

Two  walks that differ by having two consecutive equal entries $p\in\mathscr{G}$ more (or less) are said to be equivalent. The set of equivalent classes $\widetilde{\mathscr{W}}$ has a group structure induced by that of $\mathscr{W}$ and each class has a unique member (called reduced walk) without two consecutive $p\in\mathscr{G}$ \citep[][Ch. 2]{rob95}. 

Since we identify pentagonal tiles with ordered sequences $\langle x_1,x_2, x_3,x_4,x_5\rangle$, we 
act on them with the group $\widetilde{\mathscr{W}}$  as follows: $\overline{w}\bullet\langle x_1,x_2, x_3,x_4,x_5\rangle=(p_1,p_2,\dots,p_n)\bullet\langle x_1,x_2, x_3,x_4,x_5\rangle=p_n\dots p_2p_1\langle x_1,x_2,$ $ x_3,x_4,x_5\rangle$, where $(p_1,p_2,\dots,p_n)$ is the reduced walk in the class $\overline{w}$ and  the last action is the $\widetilde G_5$ action. 

In this way the  defined actions of $\widetilde{\mathscr{W}}$ and $\widetilde G_5$ on tiles are both left actions. Moreover, there is a surjective homomorphism $\varphi:\widetilde{\mathscr{W}}\rightarrow\widetilde G_5$ defined by $\varphi(p_1,p_2,\dots,p_n)=p_n\dots p_2p_1$. 

Now starting from a given tile $T$, consider the union of all translated tiles $\overline{w}\bullet T$ by reduced walks  representing the elements $\overline{w}\in\widetilde{\mathscr{W}}$, i.e. $\mathscr{U}=\bigcup_{\overline{w}\in\widetilde{\mathscr{W}}} \overline{w}\bullet T$. 
Then $\widetilde{\mathscr{W}}$ acts on $\mathscr{U}$ by shuffling the tiles of $\mathscr{U}$\footnote{In $\mathscr{U}$ all  tiles $\overline{w}\bullet T$ are different and numbered by $\widetilde{\mathscr{W}}$.}. Recalling the gluing procedure we explained in [\ref{parala:procedure}], we can endow $\mathscr{U}$ with a structure of topological space: it is a simply connected space, a tree of pentagons in which two pentagons $\overline{w_1}\bullet T$ and $\overline{w_2}\bullet T$ are glued along a side if and only if $\overline{w_2}\bullet T=p(\overline{w_1}\bullet T)$ or $\overline{w_1}\bullet T=p(\overline{w_2}\bullet T)$ for some $p\in\mathscr{G}$. Each pentagon bounds on the sides with other $5$ pentagons, and in order to visualize and fit an infinite set of pentagons together in that way we need to embed this into an infinite dimensional space.

Let $\mathscr{K}$ be the kernel of $\varphi$. Then we construct the quotient space $\widetilde{\mathscr{U}}=\mathscr{U}/\mathscr{K}$ as the set of orbits $\mathscr{K}\bullet u$, with $u\in\mathscr{U}$. This space is a Hausdorff space because the action of $\widetilde{\mathscr{W}}$ (hence that of $\mathscr{K}$) on the locally compact space $\mathscr{U}$ is discrete and proper \citep[see][Th. I.6.7]{god71}. Moreover, it has an induced Tiling by the elements of $\widetilde G_5=\Z_2\ltimes\Zt^4$: $\widetilde{\mathscr{U}}=\bigcup_{g\in\widetilde G_5} gT$ is a closed surface.

The restriction of the basis $\{T_1^2,T_1^3,T_1^4,T_1^5\}$ for $\Zt^4$ to $T$ gives a smaller group $\widetilde{G}_T\simeq\Z_2\ltimes(\Zt\oplus Ab_T)$. Namely, we are assuming that there is an exact sequence $1\rightarrow Ab_T\rightarrow \widetilde{G}_T \rightarrow \D_{12} \rightarrow 1$ and a surjective map $\widetilde G_5\overset{\pi}{\rightarrow}\widetilde{G}_T\rightarrow 1$. Now, the kernel of $\pi$ can be pulled back via $\varphi$ and give the discrete group $\mathscr{H}_T\subset \widetilde{\mathscr{W}}$. We have a new tiled closed surface $\widetilde{\mathscr{V}}_T=\mathscr{U}/\mathscr{H}_T=\bigcup_{g\in\widetilde G_T} gT$.

The group $\widetilde G_T$ acts on $\widetilde{\mathscr{V}}$ discretely and at most with a finite set of fixed points. Indeed, $\Z_2$ is generated (for instance) by the inversion $p_{12}$ which moves each tile to a different one but has the middle point of the glued tiles $gT$ and $p_{12}gT$ fixed. The translations in the normal subgroup $\Zt\oplus Ab_T$ act without fixed points\footnote{This is because the tiles associated to $T_i^j\langle a,b,c,d,e\rangle$ and $\langle a,b,c,d,e\rangle$ either are equal (in whose case the translation is trivial) or do not intersect. If they are different and have a common point they must be glued along a side by an inversion $p$, so that $pT_i^j$ fixes the tile of $\langle a,b,c,d,e\rangle$, and therefore all tiles, but this is impossible for it would mean $T_i^j=p$.}.

We are in the following situation: 
\begin{enumerate}
\item A tiled closed surface $\widetilde{\mathscr{V}}_T$ with a discrete action of $G_T$ on it.
\item An exact sequence $1\rightarrow Ab_T\rightarrow \widetilde{G}_T \rightarrow \D_{12} \rightarrow 1$ and a discrete action of $\D_{12}$ on a tiled genus $13$ surface $\Gamma$ (see [\ref{parala:icosa}]).
\end{enumerate}
We want to relate these two pieces of data\footnote{Other foundations for this situation and  the construction in [\ref{parala:construction}] are found  in the paper \citep{sin88} and bibliography therein and the book \citep{dav08}.}.

The group $Ab_T$ of translations act on $\widetilde{\mathscr{V}}_T$ properly discontinuously\footnote{Namely, each point $v\in\widetilde{\mathscr{V}}_T$ has a neighborhood  such that all its translates by $Ab_T$ are pairwise disjoint.}, thus the map $\widetilde{\mathscr{V}}_T\rightarrow\widetilde{\mathscr{V}}_T/Ab_T=\widetilde \Gamma$ is a covering map and the closed surface $\widetilde \Gamma$ is built with $24$ pentagonal tiles and the group $\D_{12}$ acts discretely on $\widetilde \Gamma$ by shuffling tiles. 

Now if we identify  fundamental pentagonal tiles on $\Gamma$ and $\widetilde \Gamma$ and give an isomorphism of the acting groups $\D_{12}$, we can extend this identification to a homeomorphism between  $\Gamma$ and $\widetilde \Gamma$ that preserves pentagonal faces, sides and vertices.

Summing-up, we can state the following:

\begin{theorem}\label{theo} Let $\langle x_1, x_2, x_3, x_4, x_5\rangle=T$ be a pitch class satisfying Condition~\ref{cond:first} and such that $\widetilde G_T$ (the restriction of $\widetilde G_5$ to the pitch class) fits into an exact sequence $1\rightarrow Ab_T\rightarrow\widetilde G_T\rightarrow \D_{12}\rightarrow 1$, where $Ab_T$ is an abelian group of order $n$. Then, there is a tiled surface $\widetilde{\mathscr{V}}$ of type $\{5,10\}$ containing all tile occurrences  of $T$ under $\widetilde G_T$ assembled so that two pentagons have a common side if and only if one is a transform of the other by an inversion in $\mathscr{G}$. The surface $\widetilde{\mathscr{V}}$ is an $n$-covering (non-branched) of a surface $\Gamma$ of genus $13$ having a Tiling by $24$ hyperbolic pentagons. The Tiling of $\widetilde{\mathscr{V}}$ has $24 n$ hyperbolic pentagons and $\text{genus}(\widetilde{\mathscr{V}})=12 n +1$.
\end{theorem}

\begin{proof} Condition~\ref{cond:first} already implies the exactness of the group short sequence. The arguments above and [\ref{parala:icosa}], [\ref{parala:construction}] show that there is a surface $\widetilde{\mathscr{V}}$ with a Tiling by pentagons and an $n$-covering $\widetilde{\mathscr{V}}\rightarrow\Gamma$ where $\Gamma$ is the genus $13$ surface which is a $2-1$ ramified cover of the Bring surface. As $\Gamma$ has a Tiling  of type $\{5,10\}$ by $24$ pentagons, then $\widetilde{\mathscr{V}}$ has $24 n$ pentagonal faces and since we have a covering map, the incidence of the pentagons at each pentagon vertex is preserved. Hurwitz formula yield $\chi(\widetilde{\mathscr{V}})=n \chi(\Gamma)=-24 n$ (because there is no ramification) and the genus Formula~\eqref{eq:genusformula}  give $g(\widetilde{\mathscr{V}})=12 n +1$, which shows the statement.
\end{proof}

\begin{example}\label{exa:five} In the pitch segment $\langle C, E, G, B\flat, D\rangle$ of example~\ref{exa:four} we found that $Ab_T\simeq\Zt$. Theorem 1 show that the surface on which the pentagonal tiles fit together has genus $g=12.12+1=145$. It is easy to see that the Stockhausen Klavierst\"uck III pentachord $\langle C, C\#, D, D\#, F\#\rangle$, the Schoenberg Op. 23/3 pentachord $\langle B\flat, D, E, B, C\#\rangle$, and the pentachords $\langle C, D, E, F, A\rangle$ and $\langle C, D, E\flat, G, A\flat\rangle$ used by A. Tcherepnin in Ops. 51, 52, and 53, also have the same pattern. Indeed, in each case one proves that $Ab_T\simeq\Zt$ and they satisfy Condition~\ref{cond:first}. Thus, these pentachords under the contextual inversions of $\widetilde G_5=G_5$ produce a Tiling on a genus $145$ surface. 
\end{example}

{\centering{
\section{Conclusion and Remarks}\label{sec:conclusion}
}}

The genus $145$ surface constructed can be viewed as a $24$ sheeted covering of the Tiling appearing in Figure~\ref{fig:Bring} (with identifications). Keeping track of the tile and sheet we are in may be kind of difficult in terms of computer graphics, however, a space representation obtained by the methods in \citep{van09} or \citep{seq07} seems to be harder due to the high genus. 

This Tonnetz representation may be useful and applicable in describing tone paths through similar pentachords using real time software that involves sound, images and media like Max/MSP/Jitter by Cycling '74\footnote{\href{http://cycling74.com/}{\tt http://cycling74.com/}.} .

A surface similar to $\Gamma$ but in the situation where one has other scales like in Koechlin or Stravinsky's Examples~\ref{exa:koechlin},~\ref{exa:stravinsky} (diatonic $\Z_7$) need to be addressed particularly. Our genus $13$ surface would not work since the dihedral group $\D_7$ is not contained in the von Dyck group $\Delta^+_{13}(2,5,10)$. So another highly symmetric surface is needed possessing  a Tiling $\{5,10\}$ by pentagons.

Also in order to model hexachords, heptachords, etc. other surfaces like $\Gamma$ with appropriate Tessellations ($\{6,6\}$, $\{7,14\}$, etc.) have to be found. Then our procedures and methods can be applied in a similar fashion.

\vspace{.5cm}

We notice that what is basically important is not precisely the Tiling of a surface but a graph containing the pitch information (or other musical information) embedded into some space (eventually a surface as we did it here). This reverts the problem of representing some kind of musical data  to that of ``Graph Drawing''\footnote{There is an active community of mathematicians and yearly proceedings on this subject \href{http://www.graphdrawing.org/}{\tt http://www.graphdrawing.org/}.} and most papers in Music Theory take this approach \citep{lew107,co03,dou98, gol98,rock09,fit11}. We pursued here the embedding of a pentachord network graph into a closed surface, which is in principle a harder problem than a singular Graph Drawing. By representing the dual Tessellations any pentachord is a vertex in the dual graph, edges relate contiguous pentachords and tiles (now decagons) close up the graph to a surface. Here a walk is a real path in this graph from vertex to vertex. This dual graph is very interesting and  a well known object in Graph Theory: the Cayley Graph\footnote{Information on this in \citep{han04}.} of our group $\widetilde G_5$ with generating set $\mathscr{G}$ (or that of the group $\widetilde G_T$ with the restricted generators).

\vspace{.5cm}

For the purpose of algebraic calculations it is easier to consider a multidimensional torus\footnote{However, we eventually lose the possibility of having a space representation.} with translations in $\Zt$ (i.e. a higher dimensional analog of the \"Ottingen-Riemann Tonnetz). This paradigm is understandable to the light of curves and their Jacobians. Any surface (Riemann surface or complete algebraic complex curve) $\Gamma$\footnote{Here I am not specifying which $\Gamma$ to take. It could be the genus $4$ Bring surface of our setting, or the genus $13$ curve in Proposition~\ref{prop:three}, or the genus $145$ curve of Example~\ref{exa:five}. Making this precise requires further developement.} of genus $g$ can be embedded into its Jacobian $Jac^g(\Gamma)$ \citep{mu77}, which is a $2 g$-dimensional real torus\footnote{$Jac^g(\Gamma)\equiv (\R/\Z)^{2 g}$} with algebraic operations. The torus where the translations $T_i^j$ occur is isogenous to the Jacobian of $\Gamma$ (i.e. either a covering or a quotient of it). So we say that in the Jacobian of $\Gamma$ one can interpret or merge both, the Tiling of the surface $\Gamma$ and the translations of the group $\widetilde G_5$, as we did  in the examples of Figures~\ref{fig:small} and~\ref{fig:tetra} for the genus $1$ case. Of course, beyond genus $1$ we do not have nice graphic representations for these tori.

\vspace{.5cm}

Assume we have a set of generators $\mathscr{G}=\{p_{1},p_{2},p_{3},p_{4},p_{5}\}$ for the group $\widetilde G_5$ so that we can write $p_{i+1}=p_1t_i$, $i=1,2,3,4$, where the $t_i$'s are translations in $\Zt$. The group $\widetilde G_5$ consists of the reflection $p_1=p$ and the translations $\{t_1,t_2,t_3,t_4\}$.  How should we encode the data for a single path in these generators? For instance $p_5p_3p_2p_4=pt_4.pt_2.pt_1.pt_3=t_4^{-1}t_2t_1^{-1}t_3$, where we have used $ptp=t^{-1}$ for any translation. Another example gives $p_5p_3p_2p_4p_5=pt_4.pt_2.pt_1.pt_3.pt_5=t_4^{-1}t_2t_1^{-1}t_3t_5^{-1}p$. We can give the following rule in case the reduced path contains only the generators in $\mathscr{G}\textbackslash\{p\}$.
\begin{enumerate}
\item if length of path is even change each occurrence of $p_{i+1}$ in an odd position by $t_i^{-1}$ and each occurrence of $p_{i+1}$ in an even position by  $t_i$. 
\item if length of path is odd do the same as in 1. but insert $p$ at the end.
\end{enumerate}
Some other examples when the path starts with $p$ and then follows by elements in $\mathscr{G}\textbackslash\{p\}$: $pp_5p_4p_3p_2=p.pt_4.pt_3.pt_2.pt_1=t_4t_3^{-1}t_2t_1^{-1}p$, or $pp_5p_4p_3=p.pt_4.pt_3.pt_2=t_4t_3^{-1}t_2$. In these cases we go by the rule.
\begin{enumerate}
\item if length of path is even change each occurrence of $p_{i+1}$ in an odd position by $t_i$, each occurrence of $p_{i+1}$ in an even position by  $t_i^{-1}$ and delete the initial $p$.
\item if length of path is odd do the same as in 1. but insert $p$ at the end.
\end{enumerate} 
In case the path has several occurrences of $p$, we separate it into segments with a starting $p$ and apply to each successive segment the rules above. An example will clarify this where we have put for short number $i$ instead of $p_i$:
$3241541451323=(324)(154)(145)(1323)=t_2^{-1}t_1t_3^{-1}p(154)(145)(1323)=t_2^{-1}t_1t_3^{-1}t_4^{-1}t_3t_3t_4^{-1}p(1323)=t_2^{-1}t_1t_3^{-1}t_4^{-1}t_3t_3t_4^{-1}t_2^{-1}t_1t_2^{-1}p=$ $t_1^2t_2^{-3}$ $t_3t_4^{-2}p$. 
 
\vspace{.5cm}
{\sc ACKNOWLEDGMENTS:} I want to thank Alejandro Iglesias Rossi for his inspiring orchestra, Mariano Fern\'andez who taught me how to use Max/MSP/Jitter and colleagues Emiliano L\'opez and Fabi\'an Quiroga for their helpful suggestions. 


\bibliographystyle{apalike}

\begin{flushleft}
{\sc Luis~A.~Piovan}\\
{\sc Departamento de Arte y Cultura\\
Viamonte esq. San Mart\'{i}n\\
 Universidad Nacional de Tres de Febrero \\
Ciudad de Buenos Aires, Argentina.\\
 and \\}
   {\sc Departamento de Matem\'atica\\
   Av. Alem 1253\\
Universidad Nacional del Sur\\
Bahía Blanca, Argentina.}\\
\href{mailto:impiovan@criba.edu.ar}
{\tt impiovan@criba.edu.ar}
\end{flushleft}


\end{document}